\documentclass[smallextended]{svjour3}                     

\smartqed  
\usepackage{geometry}                
\usepackage{graphicx}
\usepackage{amssymb}
\usepackage{epstopdf}
\usepackage{pgf,tikz}
\usepackage{amsmath}
\usepackage{tikz}
\usepackage{graphicx}
\usepackage{mathtools}
\usepackage{graphicx}
\usepackage{epstopdf}
\usepackage{fix-cm}
\usepackage{geometry}                
\usepackage{graphicx}
\usepackage{amssymb}
\usepackage{epstopdf}
\usepackage{pgf,tikz}
\usepackage{latexsym}
\usepackage{amsmath,amsfonts}
\usepackage{graphicx}
\usepackage{geometry}
\usepackage{graphicx}
\usepackage{multirow}
\usepackage{amssymb}
\usepackage{epstopdf}
\usepackage{pgf,tikz}
\usepackage{latexsym}
\usepackage[english]{babel}
\usepackage{dcolumn}
\usepackage{color}
\newcommand{\ignore}[1]{}
\newtheorem{assumption}{Assumption}

\newcommand{\oQ}{{\mathbf Q}}
\newcommand{\oR}{{\mathbb R}}
\newcommand{\vol}{\text{\rm vol}}
\newcommand{\oN}{{\mathbb N}}

\newcommand{\bK}{{\mathbf K}}

\journalname{}
\begin{document}

\title{Convergence analysis for Lasserre's measure--based hierarchy of upper bounds for polynomial optimization}


\author{Etienne de Klerk         \and
        Monique Laurent          \and
        Zhao Sun}

\institute{Etienne de Klerk \at
              Tilburg University \at
              PO Box 90153, 5000 LE Tilburg, The Netherlands \\
              \email{E.deKlerk@uvt.nl}
           \and
           Monique Laurent \at
              Centrum Wiskunde \& Informatica (CWI), Amsterdam and Tilburg University \at
              CWI, Postbus 94079, 1090 GB Amsterdam, The Netherlands\\
              \email{M.Laurent@cwi.nl}
           \and
           Zhao Sun \at
              \'Ecole Polytechnique de Montr\'eal \at
              GERAD--HEC Montreal 3000, C\^ote-Sainte-Catherine Rd, Montreal, QC H3T 2A7, Canada \\
              \email{Zhao.Sun@polymtl.ca}
}

\date{Received: date / Accepted: date}

\maketitle

\begin{abstract}
{
We consider the problem of minimizing  a continuous function $f$  over a compact set $\bK$.
We analyze a hierarchy of  upper bounds proposed by Lasserre in [{\em SIAM J. Optim.} $21(3)$ $(2011)$, pp. $864-885$],
 obtained by searching for an optimal probability density function $h$ on $\bK$ which is a sum of squares of polynomials,
 so that the expectation $\int_{\bK} f(x)h(x)dx$ is minimized.
We show that the rate of convergence is {no worse than} $O(1/\sqrt{r})$, where $2r$ is the degree bound on the density function. This analysis applies to the case when $f$ is Lipschitz continuous and   $\bK$ is a  full-dimensional
  compact set satisfying some boundary condition (which is satisfied, e.g.,  {for convex bodies}).
The $r$th upper bound in the hierarchy may be computed using semidefinite programming if $f$ is a polynomial of degree $d$,
 and if all moments of order up to $2r+d$ of the Lebesgue measure on $\bK$ are known, which holds, for example, if $\bK$ is a simplex,
  hypercube, or a Euclidean ball.
}

{
\keywords{Polynomial optimization \and Semidefinite optimization\and  Lasserre hierarchy}
\subclass{90C22\and 90C26 \and 90C30}
}
\end{abstract}

\newcommand{\fminK}{f_{\min,\mathbf{K}}}

\section{Introduction and Preliminaries}\label{secintro}

\subsection{Background}

We consider the problem of minimizing a continuous function $f:\oR^n\to\oR$ over a compact set $\mathbf{K}\subseteq \oR^n$. That is, we consider the problem of computing the parameter:

\begin{equation*}\label{fmink}
f_{\min,\mathbf{K}}:= \min_{x\in \mathbf{K}}f(x).
\end{equation*}

%
  Our main interest will be in the case where $f$ is a polynomial, and $\bK$ is defined by polynomial inequalities and equations.
    For such problems, active research has been done in recent years to construct tractable hierarchies of (upper and lower) bounds
    for $\fminK$, based on using sums of squares of polynomials and semidefinite programming (SDP).
The starting point is to reformulate $\fminK$ as the problem of finding the largest scalar $\lambda$ for which the polynomial $f-\lambda$ is nonnegative
 over $\mathbf{K}$ and then to replace the hard positivity condition by a suitable sum of squares decomposition. Alternatively, one may reformulate $\fminK$ as the
  problem of finding a probability measure $\mu$ on $K$ minimizing the integral $\int_{\mathbf{K}} fd\mu$.
These two dual points of view form the basis of the approach developed by Lasserre \cite{Las01} for building hierarchies of semidefinite programming based
 lower bounds for $\fminK$ (see  also \cite{Las09,ML09} for an overview).
 Asymptotic convergence to $\fminK$ holds (under some mild conditions on the set $\bK$).
 Moreover, error estimates have been shown in \cite{Sch,NS} when $\mathbf{K}$ is a general basic closed semi-algebraic set,  and in \cite{KL10,KLP06,KLS13,KLS14,DW,Fay,SZ14}
  for  simpler sets like the standard simplex, the hypercube and the unit sphere. In particular, \cite{Sch} shows that the rate of convergence of the hierarchy
   of lower bounds based on Schm\"udgen's Positivstellensatz is in the order $O(1/\sqrt[c]{2r})$, while
  \cite{NS} shows a convergence rate   in $O(1/\sqrt[c']{\log(2r/c') })$  for the (weaker) hierarchy of bounds based on Putinar's Positivstellensatz. Here,
   $c,c'$ are constants {(not explicitly known)} depending only on $\bK$, and $2r$ is the selected degree bound. For the case of the hypercube,  \cite{KL10} shows  (using Bernstein approximations) a convergence
    rate in $O(1/r)$ { for the lower bounds based on Schm\"udgen's Positivstellensatz}.

On the other hand, by selecting suitable probability measures on $\bK$, one obtains upper bounds for $\fminK$. This approach has been investigated,
 in particular, for minimization over  the standard simplex and  when selecting some discrete distributions over the grid points in the simplex.
The multinomial distribution is used in \cite{Nes,KLS13} to show convergence in $O(1/r)$  and  the multivariate hypergeometric distribution is used
  in  \cite{KLS14} to show convergence in $O(1/r^2)$ for quadratic minimization over the simplex (and in the general case assuming a rational minimizer exists).

Additionnally, Lasserre \cite{Las11} shows that, if we fix any measure $\mu$ on $\bK$,  then it suffices to search for a polynomial density function $h$
 which is a sum of squares  and minimizes the integral $\int_{\mathbf{K}} fhd\mu$ in order to compute the minimum $\fminK$ over $\bK$ (see Theorem \ref{thmlasnn} below).
  By adding degree constraints on the polynomial density $h$ we get a hierarchy of upper bounds for $\fminK$ and our main objective in this paper is to analyze
  the quality of this hierarchy of upper bounds for $\fminK$.
Next we will recall this result of Lasserre \cite{Las11} and then we describe our main results.

\subsection{Lasserre's hierarchy of upper bounds}
Throughout,  $\oR[x]=\oR[x_1,\dots,x_n]$ is the set of polynomials in $n$ variables with real coefficients, and $\oR[x]_r$ is the set of polynomials with degree at most $r$. $\Sigma[x]$ is the set of sums of squares of polynomials, and $\Sigma[x]_r=\Sigma[x]\cap \oR[x]_{2r}$ consists of all  sums of squares of polynomials with degree at most $2r$.
We now recall the  result of Lasserre \cite{Las11}, which  is based on the following characterization for  nonnegative continuous functions on a compact set $\mathbf K$.

\begin{theorem}\label{thmlasnn}\cite[Theorem 3.2]{Las11}
Let $\mathbf{K}\subseteq \oR^n$ be compact, let $\mu$ be an arbitrary finite Borel measure supported by $\mathbf{K}$, 
 and let $f$ be a continuous function on $\oR^n$. Then, $f$ is nonnegative on $\mathbf{K}$ if and only if
\begin{equation*}
\int_{\mathbf{K}}g^2fd\mu\ge0 \ \ \forall  g\in\oR[x].
\end{equation*}
Therefore, the minimum of $f$ over $\mathbf{K}$ can be expressed as
\begin{equation}\label{formulamu}
f_{\min,\mathbf{K}}=\inf_{h\in\Sigma[x]}\int_{\mathbf{K}}hfd\mu \ \ \text{s.t. $\int_{\mathbf{K}}hd\mu=1$.}
\end{equation}
\end{theorem}

\smallskip
\noindent
Note that formula (\ref{formulamu}) does not appear explicitly in \cite[Theorem 3.2]{Las11}, but one can derive it easily from it.
Indeed, one can write $f_{\min,\mathbf{K}}=\sup \left\{\lambda : f(x)-\lambda\ge0 \ \ \text{over $\mathbf{K}$} \right\}$. Then, by the first part of Theorem \ref{thmlasnn}, 
we have $f_{\min,\mathbf{K}}=\sup \left\{\lambda : \int_{\mathbf{K}}h(f-\lambda) d\mu\ge0\  \forall \text{$h\in\Sigma[x]$}\right\}$. As $\int_{\mathbf{K}}h(f-\lambda) d\mu=\int_{\mathbf{K}} hf  d\mu-\lambda\int_{\mathbf{K}}hd\mu$, after normalizing $\int_{\mathbf{K}}hd\mu=1$, we can conclude (\ref{formulamu}).

\smallskip\noindent
If we select the measure $\mu$ to be the Lebesgue measure in Theorem \ref{thmlasnn}, then we  obtain the following reformulation for $f_{\min,\mathbf{K}}$, which we will consider in this paper:

\begin{equation*}\label{fminkreform2}
f_{\min,\mathbf{K}}=\inf_{h\in\Sigma[x]}\int_{\mathbf{K}}h(x)f(x)dx \ \ \text{s.t. $\int_{\mathbf{K}}h(x)dx=1$.}
\end{equation*}

\smallskip
\noindent
By bounding the degree of the polynomial $h\in \Sigma[x]$ by $2r$, we can define  the parameter:

\begin{eqnarray}\label{fundr}
\underline{f}^{(r)}_{\mathbf{K}}:=\inf_{h\in\Sigma[x]_r}\int_{\mathbf{K}}h(x)f(x)dx \ \ \text{s.t. $\int_{\mathbf{K}}h(x)dx=1$.}
\end{eqnarray}

\smallskip
\noindent
Clearly,  the inequality  $f_{\min,\mathbf{K}}\le\underline{f}^{(r)}_{\mathbf{K}}$ holds for all $r\in\oN$.  Lasserre \cite{Las11} gives conditions under which  the infimum is attained in the program (\ref{fundr}). 

\begin{theorem}\cite[Theorems 4.1 and 4.2]{Las11}
Assume $\mathbf{K}\subseteq \oR^n$ is compact and has nonempty interior and let $f$ be a polynomial. Then, the program (\ref{fundr}) has an optimal solution for every $r\in\oN$ and
$$\displaystyle \lim_{r\to \infty}\underline{f}^{(r)}_{\mathbf{K}}=f_{\min,\mathbf{K}}.$$
\end{theorem}


\noindent
We now recall  how to compute the parameter $\underline{f}^{(r)}_{\mathbf{K}}$ in terms of the moments $m_\alpha(\bK)$ of the Lebesgue measure on $\bK$, where
\begin{equation*}\label{mack}
m_{\alpha}(\mathbf{K}):=\int_{\mathbf{K}}x^{\alpha}dx\ \ \ \text{ for } \alpha\in \oN^n,
\end{equation*}
and $x^{\alpha}:=\prod_{i=1}^nx_i^{\alpha_i}$.

\smallskip
\noindent
Let $N(n,r):=\{\alpha\in \oN^n: \sum_{i=1}^n\alpha_i\le r\}$, and suppose $f(x)=\sum_{\beta\in N(n,d)}f_{\beta}x^{\beta}$ has degree $d$.
If we write  $h\in\Sigma[x]_{r}$ as $h(x)=\sum_{\alpha\in N(n,2r)}h_{\alpha}x^{\alpha}$, then the parameter $\underline{f}^{(r)}_{\mathbf{K}}$ from
 (\ref{fundr}) can be reformulated as follows:
\begin{eqnarray}\label{eqSDP}
\underline{f}^{(r)}_{\mathbf{K}}&=&\min\sum_{\beta\in N(n,d)}f_{\beta}\sum_{\alpha\in N(n,2r)}h_{\alpha}m_{\alpha+\beta}(\mathbf{K})\label{fundr2}\\
 & &\text{ s.t. } \ \ \sum_{\alpha\in N(n,2r)}h_{\alpha}m_{\alpha}(\mathbf{K})=1,\nonumber\\
&&\ \ \ \ \ \ \ \sum_{\alpha\in N(n,2r)}h_{\alpha}x^{\alpha}\in\Sigma[x]_r.\nonumber
\end{eqnarray}

\smallskip\noindent
Hence, if we know the moments $m_{\alpha}(\mathbf{K})$ for all $\alpha\in\oN^n$ with $|\alpha|:=\sum_{i=1}^n\alpha_i \le d+2r$,
 then we can compute the parameter  $\underline{f}^{(r)}_{\mathbf{K}}$ by solving the semidefinite program
  (\ref{eqSDP}) which involves a LMI of size $n+2r\choose 2r$.
{So the bound $\underline{f}^{(r)}_{\mathbf{K}}$ can be computed in polynomial time for fixed $d$ and $r$ (to any fixed precision).}

\smallskip\noindent
When $\mathbf{K}$ is the standard  simplex $\Delta_n=\{x\in\oR^n_+:\sum_{i=1}^n x_i\le1\}$, the unit hypercube $\oQ_n=[0,1]^n$,  or the unit ball $B_1(0)=\{x\in \oR^n: \|x\|\le 1\}$,  there exist explicit formulas for the moments $m_{\alpha}(\mathbf{K})$.
Namely, for the standard simplex, we have
\begin{equation}\label{mc0}
m_{\alpha}(\Delta_n)={\prod_{i=1}^n\alpha_i!\over (|\alpha| +n)!},
\end{equation}
see e.g., \cite[equation (2.4)]{LZ01} or \cite[equation (2.2)]{GM78}. From this one  can easily calculate the moments for the hypercube $\oQ_n$: 
\begin{eqnarray*}\label{galphac}
m_{\alpha}(\oQ_n)=\int_{\oQ_n}x^{\alpha}dx=\prod_{i=1}^n \int_0^1x_i^{\alpha_i}dx_i=\prod_{i=1}^n \frac{1}{\alpha_i+1}.
\end{eqnarray*}

\smallskip
\noindent
To state the moments for the unit Euclidean ball, we will use the notation $[n]:= \{1,\ldots,n\}$, the Euler gamma function $\Gamma(\cdot)$, and the notation for the double factorial of an integer $k$:
\begin{eqnarray*}
k!! = \left\{ \begin{array}{ll}
k\cdot(k-2)\cdots 3\cdot1, & \textrm{if $k>0$ is odd,}\\
k\cdot(k-2)\cdots 4\cdot2, & \textrm{if $k>0$ is even,}\\
1 & \textrm{if $k=0$ or $k=-1$.}
\end{array} \right.
\end{eqnarray*}
In terms of this notation, the moments for the unit Euclidean ball are given by:
\begin{eqnarray}\label{malphaball}
m_{\alpha}(B_1(0)) = \left\{ \begin{array}{ll}
\frac{\pi^{n/2}\prod_{i=1}^n\left(\alpha_i-1\right)!!}{\Gamma\left(1+{n+|\alpha|\over 2}\right)2^{|\alpha|/2}} =
{\pi^{(n-1)/2}2^{(n+1)/ 2} \prod_{i=1}^n\left(\alpha_i-1\right)!! \over (n+|\alpha|)!!}
&    \textrm{\quad if $\alpha_i$ is even for all $i\in[n]$,}\\
0 & \textrm{\quad otherwise.}
\end{array} \right.
\end{eqnarray}
One may prove relation \eqref{malphaball} using
\[
\int_{B_1(0)} x^\alpha dx = \frac{1}{\Gamma(1+(n+|\alpha|)/2)}\int_{\mathbb{R}^n} x^\alpha\mbox{exp}\left(-\|x\|^2\right)dx
\]
(see, e.g.,\ \cite[Theorem 2.1]{Las14}), together with the fact (see, e.g., page $872$ in \cite{Las11}) that
 \begin{eqnarray*}
\int_{-\infty}^{+\infty}t^p \exp\left(-t^2/2\right)dt = \left\{ \begin{array}{ll}
\sqrt{2\pi}(p-1)!! & \textrm{\quad if $p$ is even,}\\
0 & \textrm{\quad if $p$ is odd,}
\end{array} \right.
\end{eqnarray*}
and the identity $\Gamma(1+{k\over 2})={k!!\over 2^{(k+1)/2}} \sqrt \pi$ for all integers $k\in\oN$ (see e.g., \cite[Section 6.1.12]{AS64}).

\smallskip\noindent
For a general  polytope $\mathbf{K}\subseteq\oR^n$, it is a hard problem to compute the moments $m_{\alpha}(\mathbf{K})$.
In fact, the problem of computing the volume of polytopes of varying dimensions is already \#P-hard \cite{DF88}.
On the other hand,
 any polytope $\mathbf{K}\subseteq\oR^n$  can be triangulated into finitely many simplices (see e.g., \cite{LRS08}) so that one  could  use (\ref{mc0}) to obtain the moments $m_{\alpha}(\mathbf{K})$ of $\bK$. The complexity of this method depends on the number of simplices in the triangulation.
However, this number can be exponentially large (e.g., for the hypercube) and the problem of finding  the smallest possible triangulation of a polytope is NP-hard, even in fixed dimension $n=3$ (see e.g., \cite{LRS08}).

\subsubsection*{Example}
Consider the minimization of the Motzkin polynomial $f(x_1,x_2)=x_1^4x_2^2+x_1^2x_2^4-3x_1^2x_2^2+1$ over the hypercube $\bK = [-2,2]^2$, which has four global minimizers at the points $(\pm 1,\pm 1)$, and $f_{\min,\bK} = 0$.
Figure \ref{figure:motzkin18} shows the computed optimal sum of squares density function $h^*$, for  $r=12$, corresponding to $\underline{f}^{(12)}_\bK = 0.406076$. We observe that the optimal density $h^*$ shows four peaks at the four global minimizers and thus,   it  appears to approximate  the density of a  convex combination of the Dirac measures at the four minimizers.

\begin{figure}[h!]
\begin{center}
\includegraphics[width=0.45\textwidth]{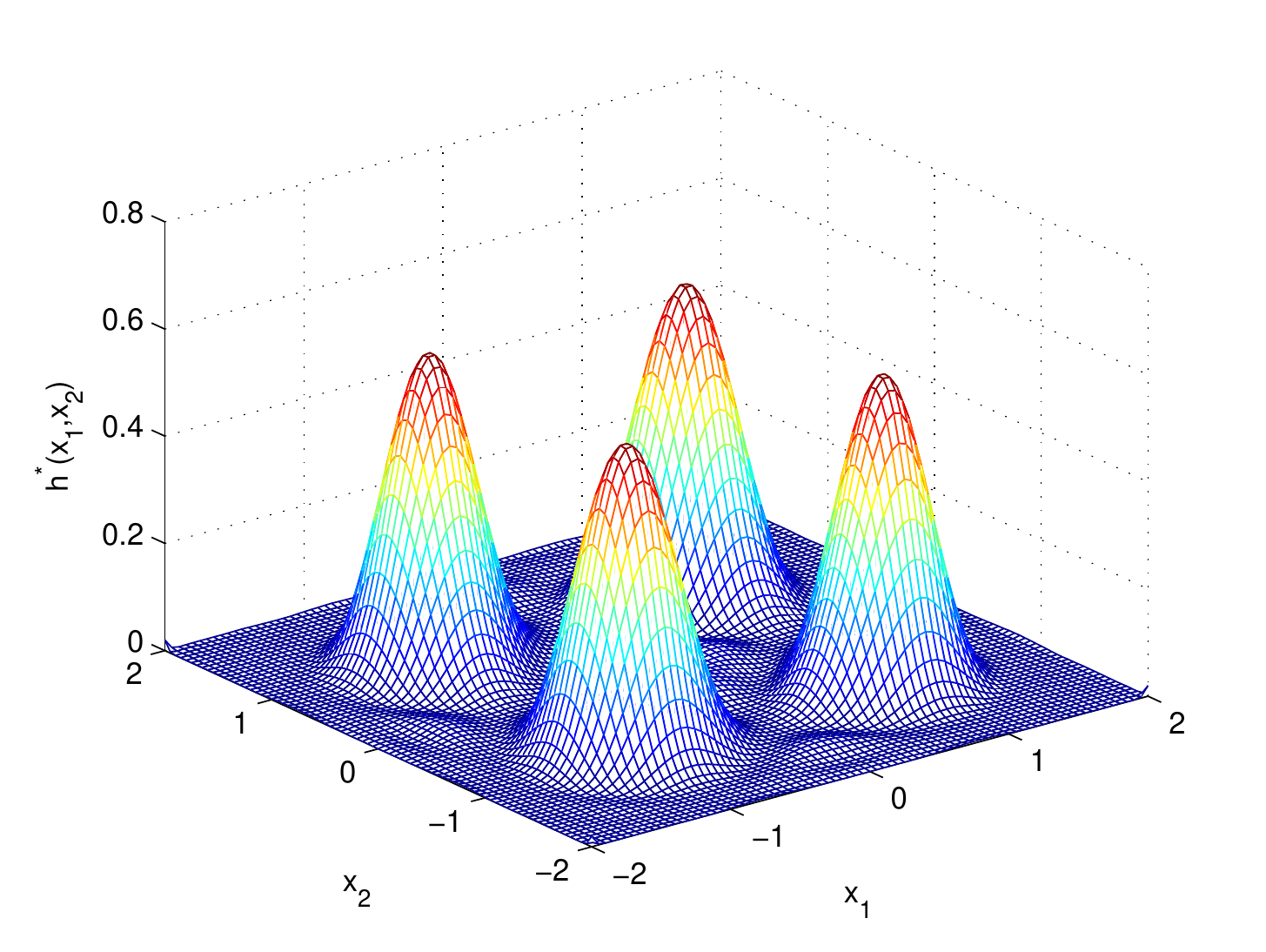}
\includegraphics[width=0.45\textwidth]{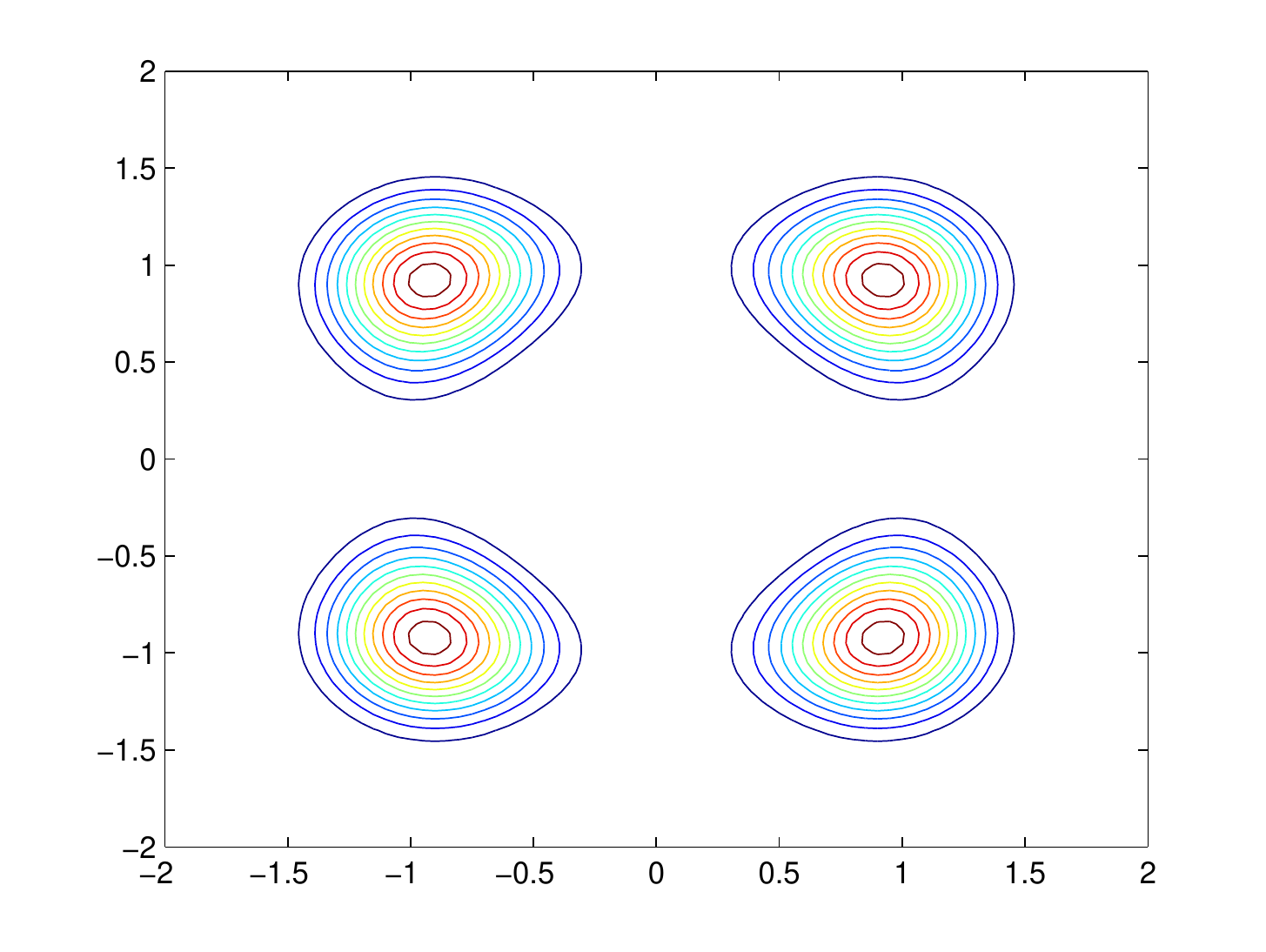}\\
  \caption{\label{figure:motzkin18}Graph and contour plot of $h^*(x)$ on $[-2,2]^2$ ($r=12$ and $\deg(h^*)=24$) for the Motzkin polynomial.}
\end{center}
\end{figure}

\noindent We will present several additional numerical examples in Section \ref{sec:numerial examples}.

\subsection{Our main results}
In this paper we analyze the quality of the upper bounds $\underline{f}^{(r)}_{\mathbf{K}}$ from (\ref{fundr}) for the minimum $\fminK$ of $f$ over $K$.
Our main result  is an  upper bound for the range $\underline{f}^{(r)}_{\mathbf{K}}-f_{\min,\mathbf{K}}$, which applies to the case when $f$
 is Lipschitz continuous on $\mathbf{K}$ and when $\mathbf{K}$ is a full-dimensional compact set satisfying the
   additional  condition  from  Assumption \ref{propeta}, see Theorem \ref{thmmain} below.
We will use throughout the following notation about the set $\mathbf{K}$.

We let $D(\mathbf{K})=\max_{x,y\in\mathbf{K}} \|x-y\|^2$ denote the {(squared)} diameter of the set $\bK$, where $\|x\|=\sqrt{\sum_{i=1}^n{x_i}^2}$ is the $\ell_2$-norm.
Moreover, $w_{\min}(\mathbf{K})$ is the minimal width of $\mathbf{K}$, which is the minimum distance between two distinct parallel supporting hyperplanes of $\mathbf{K}$.
Throughout, $B_\epsilon(a):=\{x\in \oR^n: \|x-a\|\le \epsilon\}$ denotes the Euclidean ball centered at $a\in\oR^n$ and with radius $\epsilon>0$. With $\gamma_n$ denoting the volume of the $n$-dimensional unit ball, the volume of the ball $B_\epsilon(a)$ is given by
$\text{vol} B_\epsilon (a)= \epsilon^n\gamma_n.$

\smallskip
{We now formulate our geometric assumption about the set $\mathbf K$ which says (roughly) that around any point $a\in \mathbf K$ there is a ball intersecting a constant fraction of the unit ball.}

\begin{assumption}\label{propeta}
{For all points $a\in \mathbf K$ there exist constants $\eta_{\mathbf{K}}>0$ and ${\epsilon}_{\mathbf{K}}>0$ such that}
\begin{equation}\label{volballn}
\vol (B_\epsilon(a)\cap \mathbf{K}) \ge {\eta_{\mathbf{K}}}\vol B_\epsilon(a)={\eta_{\mathbf{K}}}\epsilon^n\gamma_n \ \ \text{for all $0<\epsilon\le{\epsilon}_{\mathbf{K}}$}.
\end{equation}
\end{assumption}

\smallskip\noindent
{Note that Assumption \ref{propeta} implies that the set $\mathbf K$ has positive Lebesgue density at all points $a\in \mathbf K$.}
For all sets $\mathbf{K}$ satisfying Assumption~\ref{propeta}, we also define the parameter
\begin{equation}\label{addefrK}
{r_\bK:= \max\left\{ {D(\mathbf{K})e\over 2\epsilon_{\bK}^{3}}, {n}\right\} \ \text{ if } \epsilon_\bK \le 1, \ \text{ and }\
r_\bK:={D(\bK)e\over 2}  \ \text{ if } \epsilon_\bK \ge 1.}
\end{equation}
 {Here, $e=2.71828...$ denotes the base of the natural logarithm.}
{Note that the parameters $\eta_{\mathbf K}$, $\epsilon_{\mathbf K}$ and $r_{\mathbf K}$ depend not only on the set $\mathbf K$ but also on the point $a\in \mathbf K$; we omit the dependance on $a$  to simplify notation.
Assumption \ref{propeta} will be used in the case when the point $a$ is a global minimizer in $\mathbf K$ of the polynomial to be analyzed. }

\smallskip\noindent
{For instance,  convex bodies and, more generally, compact star-shaped sets satisfy Assumption \ref{propeta} (see   Section \ref{secass}).
  We now give an example of a set $\mathbf K$ that does not satisfy Assumption \ref{propeta} and refer to Section \ref{secass} for more discussion about Assumption \ref{propeta}.}

\begin{example}
{Consider the following set $\bK\subseteq\oR^2$,  displayed  in Figure \ref{countereg}:
$$\bK=\{x\in\oR^2\ :\ x\ge0, (x_1-1)^2+(x_2-1)^2\ge 1\}.$$
One can easily check that Assumption \ref{propeta} is not satisfied, since the condition \eqref{volballn} does not hold for the two points $a$ and $b$.}
\end{example}

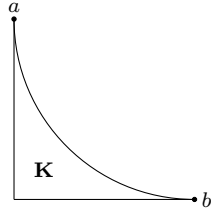
\begin{figure}
\centering
\begin{tikzpicture}[scale=0.8]
\centering
\draw  (-3,-3)--(-3,0);
\draw  (-3,-3)--(0,-3);
\draw (0,-3) arc (-90:-180:3cm);
\filldraw [black] (-3,0) circle (1pt);
\filldraw [black] (0,-3) circle (1pt);
\draw (-3,0.2) node {$a$};
\draw (0.2,-3) node {$b$};
\draw (-2.5,-2.5) node {$\bK$};
\end{tikzpicture}
\caption[Two-dimensional Euclidean ball\index{Euclidean ball}]{This set $\bK$ does not satisfy Assumption \ref{propeta} at the points $a$ and $b$.}
\label{countereg}
\end{figure}

\smallskip\noindent
We  now present our main result.

\begin{theorem}\label{thmmain}
Assume that $\mathbf{K}\subseteq\oR^n$ is compact and satisfies Assumption \ref{propeta}.
Then there exists a constant $\zeta(\mathbf{K})$ (depending only on $\bK$) such that, for all Lipschitz continuous functions $f$ with Lipschitz constant $M_f$ on $\mathbf{K}$, the following inequality holds:
\begin{equation}\label{thmmaineq1}
\underline{f}^{(r)}_{\mathbf{K}}-f_{\min,\mathbf{K}}\le {\zeta(\mathbf{K}) M_f\over \sqrt{r}} \ \ \text{ for all  } r\ge r_\bK+1.
\end{equation}

\smallskip\noindent
Moreover, if $f$ is a polynomial of degree $d$  and $\mathbf{K}$ is a convex body, then
\begin{equation}\label{thmmaineq2}
\underline{f}^{(r)}_{\mathbf{K}}-f_{\min,\mathbf{K}}\le  {2d^2 \zeta(\mathbf{K})\sup_{x\in\mathbf{K}}|f(x)|\over w_{\min}(\mathbf{K})} {1\over \sqrt r}\ \ \text{ for all  } r\ge r_\bK+1.
\end{equation}
\end{theorem}

\noindent
The key idea to show this result is to select suitable sums of squares densities which we are able to analyse. For this, we will select a global minimizer $a$ of $f$ over $\mathbf{K}$ and consider the Gaussian distribution with mean $a$ and, as sums of squares densities, we will select the polynomials $H_{r,a}$ obtained  by truncating  the Taylor series expansion of the Gaussian distribution, see relation (\ref{hrax}).

{\begin{remark}
{When the polynomial $f$ has a root in $\mathbf{K}$} (which can be assumed without loss of generality),
  the parameter
$\sup_{x\in \mathbf K} |f(x)|$ involved in relation (\ref{thmmaineq2})  can easily be upper bounded in terms of the range of values of $f$; namely,
$$\sup_{x\in \mathbf K} |f(x)| \le f_{\max,\mathbf K}-f_{\min,\mathbf K},$$
where $f_{\max,\mathbf K}$ denotes the maximum value of $f$ over $\mathbf K$.
Hence relation (\ref{thmmaineq2}) also implies an upper bound on $\underline{f}^{(r)}_{\mathbf{K}}-f_{\min,\mathbf{K}}$
in terms of the range $ f_{\max,\mathbf K}-f_{\min,\mathbf K},$ as is commonly used in approximation analysis (see, e.g., \cite{KHE08,KLP06}).
\end{remark}
}



\subsection{Contents of the paper}
Our paper is organized as follows. In Section \ref{secpfmain}, we give a constructive proof for our main result in Theorem \ref{thmmain}.
In Section \ref{secgenerate} we show how to obtain feasible points in $\bK$ that correspond to the bounds $\underline{f}^{(r)}_{\mathbf{K}}$ through sampling.
This is followed by a section with numerical examples (Section \ref{sec:numerial examples}).
 {Finally, in the concluding remarks (Section \ref{sec:conclusion}), we revisit Assumption \ref{propeta}, and discuss
  perspectives for future research.
}

\section{Proof of our main result in Theorem \ref{thmmain}}\label{secpfmain}
In this section we prove our main result in Theorem \ref{thmmain}.
 Our analysis holds for Lipschitz continuous functions, so we  start by reviewing some relevant properties in Section \ref{seclipcf}.
  In the next step we indicate in Section \ref{secmr} how to select the polynomial density function $h$ as a special sum of squares that we will be able to analyze. Namely,  we let  $a$ denote a global minimizer  of the function $f$ over the set $\mathbf{K}\subseteq\oR^n$.
Then we consider the density function $G_a$ in (\ref{Ga}) of the Gaussian distribution with mean  $a$  {(and suitable variance)} and  the polynomial $H_{r,a}$ in (\ref{hrax}), which is obtained from the truncation at degree $2r$ of the Taylor series expansion of the Gaussian density function $G_a$.
The final  step will be to analyze the quality of the bound obtained by selecting the polynomial $H_{r,a}$ and this will be the most technical part of the proof, carried out in Section \ref{secthmpf}.

\subsection{Lipschitz continuous functions}\label{seclipcf}

A function $f$ is said to be Lipschitz continuous on $\mathbf{K}$, with Lipschitz constant $M_f$, if it satisfies:
\begin{equation*}\label{ineqmf1}
|f(y)-f(x)|\le M_f\|y-x\| \quad \text{ for all } x,y\in \mathbf K.
\end{equation*}
If $f$ is continuous and differentiable on $\mathbf{K}$, then $f$ is Lipschitz continuous on $\mathbf{K}$ with respect to the constant
\begin{equation}\label{formf}
M_f=\max_{x\in\mathbf{K}}\|\nabla f(x)\|.
\end{equation}
Furthermore, if $f$ is an $n$-variate polynomial with degree $d$, then the Markov inequality for $f$ on a convex body $\mathbf{K}$ reads as
\begin{equation*}
\max_{x\in\mathbf{K}}\|\nabla f(x)\|\le\frac{2d^2}{w_{\min}(\mathbf{K})}\sup_{x\in\mathbf{K}}|f(x)|,
\end{equation*}
see e.g., \cite[relation (8)]{KHE08}. Thus, together with (\ref{formf}), we have that $f$ is Lipschitz continuous on $\mathbf{K}$ with respect to the constant
\begin{equation}\label{ineqmf}
M_f\le\frac{2d^2}{w_{\min}(\mathbf{K})}\sup_{x\in\mathbf{K}}|f(x)|.
\end{equation}

\subsection{Choosing the polynomial density function $H_{r,a}$}\label{secmr}

Consider the function
\begin{equation}\label{Ga}
G_a(x):={1\over (2\pi \sigma^2)^{n/2}}\exp\left(-{\|x-a\|^2\over 2\sigma^2}\right),
\end{equation}
which is the probability density function of the Gaussian distribution with mean $a$ and standard variance $\sigma$ (whose value will be defined later).
Let the constant $C_{\mathbf{K},a}$ be defined by
\begin{equation}\label{cka}
\int_{\mathbf{K}}C_{\mathbf{K},a}G_a(x)dx=1.
\end{equation}

\smallskip\noindent
Observe that $G_a(x)$ is equal to the function ${1\over (2\pi \sigma^2)^{n/2}}e^{-t}$ evaluated at the point $t={\|x-a\|^2\over 2\sigma^2}$.

\noindent Denote by $H_{r,a}$ the Taylor series expansion of $G_a$ truncated at the order $2r$. That is,
\begin{equation}\label{hrax}
H_{r,a}(x)={1\over (2\pi \sigma^2)^{n/2}}\sum_{k=0}^{2r}{1\over k!}\left(-{\|x-a\|^2\over 2\sigma^2}\right)^k.
\end{equation}
Moreover consider the constant $c^r_{\mathbf{K},a}$, defined by
\begin{equation}\label{ckra}
\int_{\mathbf{K}}c^r_{\mathbf{K},a}H_{r,a}(x)dx=1.
\end{equation}
The next step is to show that  $H_{r,a}$ is a sum of squares of polynomials and thus $H_{r,a}\in\Sigma[x]_{2r}$.
This follows from the next lemma.

\begin{lemma}\label{lemf2rsos}
Let $\phi_{2r}(t)$ denote the (univariate) polynomial of degree $2r$ obtained by truncating the Taylor series expansion of $e^{-t}$  at the order $2r$. That is,
\begin{equation*}\label{phi2rt}
\phi_{2r}(t):=\sum_{k=0}^{2r}{(-t)^k\over k!}.
\end{equation*}
Then  $\phi_{2r}$ is a sum of squares of polynomials.
Moreover, we have
\begin{equation}\label{fmf2r}
0\le \phi_{2r}(t) - e^{-t} \le {t^{2r+1}\over (2r+1)!} \quad \text{ for all } t\ge 0.
\end{equation}
\end{lemma}

\begin{proof}
First, we show that $\phi_{2r}$ is a sum of squares. As   $\phi_{2r}$ is a univariate polynomial, by Hilbert's Theorem (see e.g., \cite[Theorem 3.4]{ML09}), it suffices to show that  $\phi_{2r}(t)\ge0$ for all $t\in\oR$.
As $\phi_{2r}(-\infty)=\phi_{2r}(+\infty) =+\infty$, it suffices to show that $\phi_{2r}(t)\ge 0$ at all the stationary points $t$ where $\phi_{2r}'(t)=0$.
For this,  observe that $\phi_{2r}'(t)=\sum_{k=1}^{2r}(-1)^k{t^{k-1}\over (k-1)!},$
so that it can be written as $\phi_{2r}'(t)=-\phi_{2r}(t)+{t^{2r}\over (2r)!}.$ Hence, for all $t$ with $\phi_{2r}'(t)=0$, we have $\phi_{2r}(t)={t^{2r}\over (2r)!}\ge0$.

\smallskip\noindent
Next, we show that $\phi_{2r}(t)\ge e^{-t}$ for all $t\ge 0$. Fix $t\ge 0$. Then, by Taylor Theorem (see e.g., \cite{WW96}), one has $e^{-t}=\phi_{2r}(t)+\frac{\phi^{(2r+1)}(\xi)t^{2r+1}}{(2r+1)!}$ for some $\xi\in [0,t]$. As $\phi^{(2r+1)}(\xi)=-e^{-\xi}$, one can conclude
that $e^{-t}-\phi_{2r}(t)=  - {e^{-\xi}t^{2r+1}\over (2r+1)!} \le 0$ and $e^{-t}-\phi_{2r}(t)\ge -{t^{2r+1}\over (2r+1)!}.$
\qed
\end{proof}

\smallskip\noindent

\noindent We now consider the parameter $f^{(r)}_{\mathbf{K},a}$ defined as
\begin{eqnarray}
f^{(r)}_{\mathbf{K},a}:=\int_{\mathbf{K}}f(x)c^r_{\mathbf{K},a}H_{r,a}(x)dx.\label{frk}
\end{eqnarray}

\smallskip\noindent
Our main technical result is the following upper bound  for the range $f^{(r)}_{\mathbf{K},a}-f_{\min,\mathbf{K}}$.

\begin{theorem}\label{thmfrk}
Assume  $\mathbf{K}\subseteq\oR^n$ is compact and satisfies Assumption \ref{propeta}, and consider the parameter  $r_\bK$ from (\ref{addefrK}).
Then there exists a constant $\zeta(\mathbf{K})$ (depending only on $\bK$) such that, for all Lipschitz continuous functions $f$ with Lipschitz constant $M_f$ on $\mathbf{K}$, the following inequality holds:
\begin{equation}\label{thmfrkeq1}
f^{(r)}_{\mathbf{K},a}-f_{\min,\mathbf{K}}\le {\zeta(\mathbf{K}) M_f\over \sqrt{2r+1}},\ \ \text{ for all  } r\ge {r_\bK\over 2}.
\end{equation}

\smallskip\noindent
Moreover, if $f$ is a polynomial of degree $d$ and $\mathbf{K}$ is a convex body, then
\begin{equation}\label{thmfrkeq2}
f^{(r)}_{\mathbf{K},a}-f_{\min,\mathbf{K}}\le  {2d^2 \zeta(\mathbf{K})\sup_{x\in\mathbf{K}}|f(x)|\over w_{\min}(\mathbf{K})\sqrt{2r+1}},\ \ \text{ for all  } r\ge {r_\bK\over 2}.
\end{equation}
\end{theorem}

\medskip\noindent
{We will give the proof of Theorem \ref{thmfrk}, which has lengthy technical details,  in Section \ref{secthmpf} below. We now show how to derive Theorem \ref{thmmain}  as a direct application of Theorem \ref{thmfrk}.}

\smallskip\noindent
\begin{proof}{\em (of Theorem \ref{thmmain})}
Assume $f$ is Lipschitz continuous with Lipschitz constant $M_f$ on $K$ and  $a$ is a minimizer of $f$ over the set $\mathbf{K}$.
Using the definitions (\ref{fundr}) 
and (\ref{frk}) of the parameters and the fact that
 $H_{r,a}$ is a sum of squares with degree $4r$, it follows that
\begin{equation*}
\underline{f}^{(2r+1)}_{\mathbf{K}}\le\underline{f}^{(2r)}_{\mathbf{K}}\le f^{(r)}_{\mathbf{K},a},\ \ \text{for all $r\in\oN$}.
\end{equation*}
Then, from inequality (\ref{thmfrkeq1}) in Theorem \ref{thmfrk}, one obtains
\begin{equation*}
\underline{f}^{(2r+1)}_{\mathbf{K}}-f_{\min,\mathbf{K}}\le\underline{f}^{(2r)}_{\mathbf{K}}-f_{\min,\mathbf{K}}\le f^{(r)}_{\mathbf{K},a}-f_{\min,\mathbf{K}}\le {\zeta(\mathbf{K}) M_f\over \sqrt{2r+1}}\ \ \text{ for all  } r\ge {r_\bK\over 2}.
\end{equation*}

\smallskip\noindent
Hence, for all $r\ge {r_\bK}+1$,
\begin{eqnarray*}
\underline{f}^{(r)}_{\mathbf{K}}-f_{\min,\mathbf{K}} &\le& {\zeta(\mathbf{K}) M_f\over \sqrt{r+1}}\le {\zeta(\mathbf{K}) M_f\over \sqrt{r}} \ \ \text{for even $r$,}\\
\underline{f}^{(r)}_{\mathbf{K}}-f_{\min,\mathbf{K}} &\le& {\zeta(\mathbf{K}) M_f\over \sqrt{r}} \ \ \text{for odd $r$.}
\end{eqnarray*}

\smallskip\noindent
This concludes the proof for relation (\ref{thmmaineq1}),  and relation  (\ref{thmmaineq2}) follows  from (\ref{thmfrkeq2}) in an analogous way. This finishes the proof of Theorem \ref{thmmain}.
\qed
\end{proof}

\subsection{Analyzing the polynomial density function $H_{r,a}$}\label{secthmpf}

In this section we prove the result of Theorem \ref{thmfrk}. Recall that $a$ is a global minimizer of $f$ over $\mathbf{K}$.
For the proof, we will need the following four technical lemmas.

\begin{lemma}\label{lemcra}
Assume $\mathbf{K}\subseteq\oR^n$ is compact and satisfies Assumption \ref{propeta}.
Then, for all $0<\epsilon\le\epsilon_{\mathbf{K}}$ and $r\in\oN$, we have:
\begin{equation}\label{ineqcrad}
c^r_{\mathbf{K},a}\le C_{\mathbf{K},a}\le \frac{(2\pi\sigma^2)^{n/2}\exp\left({\epsilon^2\over 2\sigma^2}\right)}{\eta_{\mathbf{K}} \epsilon^n \gamma_n}.
\end{equation}
\end{lemma}

\begin{proof}
By Lemma \ref{lemf2rsos}, $\phi_{2r}(t)\ge e^{-t}$ for all $t\ge 0$, which implies $H_{r,a}(x)\ge G_a(x)$ for all $x\in \oR^n$.
Together with the relations (\ref{cka}) and (\ref{ckra}) defining the constants $C_{\mathbf{K},a}$ and $c^r_{\mathbf{K},a}$, we deduce that  $c^r_{\mathbf{K},a}\le C_{\mathbf{K},a}$. Moreover, by the definition (\ref{cka}) of the constant $C_{K,a}$, one has
\begin{eqnarray*}
{1\over C_{\mathbf{K},a}}&=&\int_{\mathbf{K}}G_a(x)dx=\int_{\mathbf{K}}{1\over (2\pi \sigma^2)^{n/2}}\exp\left(-{\|x-a\|^2\over 2\sigma^2}\right)dx\\
&\ge& \int_{\mathbf{K}\cap B_{\epsilon}(a)}{1\over (2\pi \sigma^2)^{n/2}}\exp\left(-{\|x-a\|^2\over 2\sigma^2}\right)dx\\
&\ge& {1\over (2\pi \sigma^2)^{n/2}}\exp\left(-{\epsilon^2\over 2\sigma^2}\right)\vol(\mathbf{K}\cap B_\epsilon(a)).
\end{eqnarray*}

\smallskip
\noindent
We now use relation (\ref{volballn}) from Assumption \ref{propeta} in order to conclude that  $\vol(\mathbf{K}\cap B_\epsilon(a))\ge {\eta_{\mathbf{K}}}\epsilon^n\gamma_n$, which gives the desired upper bound on $C_{K,a}$.
\qed
\end{proof}

\begin{lemma}\label{lemgf}
Given $\tilde{x}\in\oR^n$ and a function $F:\oR_+\to\oR$, define the function $f:\oR^n\to\oR$ by  $f(x)=F(\|x-\tilde{x}\|)$  for all $x\in\oR^n$. Then, for all $\rho_2\ge\rho_1\ge 0$, one has
\begin{eqnarray*}
\int_{B_{\rho_2}(\tilde{x})\backslash B_{\rho_1}(\tilde{x})}f(x)dx=n\gamma_n\int^{\rho_2}_{\rho_1}z^{n-1}F(z)dz,
\end{eqnarray*}
where $\gamma_n=\frac{\pi^{(n-1)/2}2^{(n+1)/2}}{n!!}$ is the volume of the unit Euclidean ball in $\oR^n$.
\end{lemma}

\begin{proof}
Apply a change of variables using spherical coordinates as explained, e.g., in  \cite{LEB60}.
\qed
\end{proof}

\smallskip\noindent
\begin{lemma}\label{lemadd1}
For all positive integers $r$ and $n$, one has $\left({1\over 2r+1}\right)^{-{n\over 4(2r+1)+2n}}< 6n$.
\end{lemma}
\begin{proof}
Let $n \in \mathbb{N}$ be given.
Denote $$g(r):=\left({1\over 2r+1}\right)^{-{n\over 4(2r+1)+2n}}=\left( 2r+1 \right)^{{n\over 4(2r+1)+2n}} \quad (r \ge 0).$$
Observe that, $g(0) = 1$, $g(r)>0$ for all $r \ge 0$, $\ln(g(r))={n\over 8r+4+2n}\ln(2r+1)$, and thus $\lim_{r\rightarrow \infty} g(r)=1$. It suffices to show $g(r^*)< 6n$ for all stationary points $r^*$.
Since
\begin{eqnarray*}
{d\ln(g(r))\over dr}={-8n\ln(2r+1)\over (8r+4+2n)^2}+{2n\over (2r+1)(8r+4+2n)},
\end{eqnarray*}
and $g'(r) = \frac{1}{g(r)}{d\ln(g(r))\over dr}$, any stationary point $r^*$ satisfies
\begin{eqnarray*}
{d\ln(g(r^*))\over dr} = 0 \Longleftrightarrow (2r^*+1)\left[ \ln(2r^*+1)-1 \right]={n\over 2}.
\end{eqnarray*}
Since
\begin{eqnarray*}
(2r^*+1)(\ln(3)-1)\le (2r^*+1)\left[ \ln(2r^*+1)-1 \right]={n\over 2},
\end{eqnarray*}
one has $2r^*+1 \le {n\over 2 (\ln(3)-1)}<6n$.
Since $g(r) \le 2r+1$ for all $r\ge 0$, one has $g(r^*)\le 2r^*+1 < 6n.$
\qed
\end{proof}

\smallskip\noindent
\begin{lemma}\label{propga}
Assume $\mathbf{K}\subseteq\oR^n$ is compact and satisfies Assumption \ref{propeta}.
Then, for all $0< \epsilon \le \epsilon_{\mathbf{K}}$, one has
\begin{equation*}\label{thmfrkeq3}
\int_{\mathbf{K}}C_{\mathbf{K},a}\|x-a\|G_a(x)dx \le \epsilon+{n\sigma^{n+1}p(n)\over \epsilon^n\eta_{\mathbf{K}}}e^{{\epsilon^2\over 2\sigma^2}},
\end{equation*}
where $p(n):=\int_0^{+\infty} t^n e^{-t^2/2}dt$ is a constant depending on $n$, given by
\begin{eqnarray}\label{constantpn}
p(n) = \left\{ \begin{array}{ll}
1 & \textrm{if $n=1$,}\\
\sqrt{{\pi\over 2}}\prod_{j=1}^k\left(2j-1\right) & \textrm{if $n=2k$ and $k\ge1$,}\\
\prod_{j=1}^k\left(2j\right) & \textrm{if $n=2k+1$ and $k\ge1$.}
\end{array} \right.
\end{eqnarray}

\end{lemma}

\begin{proof}
Let  $\varphi:=\int_{\mathbf{K}}C_{\mathbf{K},a}\|x-a\|G_a(x)dx$ denote the integral that we need to upper bound.
We split the integral $\varphi$ as $\varphi=\varphi_1+\varphi_2$, depending on whether $x$ lies in the ball $B_\epsilon(a)$ or not.

\smallskip
\noindent
First, we upper bound the term $\varphi_1$ as
$$\varphi_1:=\int_{\mathbf{K}\cap B_\epsilon(a)} \|x-a\| C_{\mathbf{K},a} G_a(x)dx \le \epsilon \int_{\mathbf{K}\cap B_\epsilon(a)}  C_{\mathbf{K},a} G_a(x)dx
\le \epsilon \int_{\mathbf{K}}  C_{\mathbf{K},a} G_a(x)dx =\epsilon.$$
Second, we bound the integral
$$\varphi_2:=C_{\mathbf{K},a}\int_{\mathbf{K}\setminus B_\epsilon(a)} \|x-a\| G_a(x)dx.$$
Since $\mathbf{K}\subseteq B_{\sqrt{D(\mathbf{K})}}(a)$, one has
\begin{equation*}
\varphi_2\le C_{\mathbf{K},a}\int_{B_{\sqrt{D(\mathbf{K})}}(a)\setminus B_\epsilon(a)} \|x-a\| G_a(x)dx,
\end{equation*}
where the right hand side, by Lemma \ref{lemgf}, is equal to
$${C_{\mathbf{K},a}n\gamma_n\over (2\pi\sigma^2)^{n/2}}\int_{\epsilon}^{\sqrt{D(\mathbf{K})}}z^n\exp\left(-{z^2\over 2\sigma^2}\right) dz.$$
By a change of variable $t={z\over\sigma}$, one obtains
$$\varphi_2 \le {C_{\mathbf{K},a}n\gamma_n\sigma\over (2\pi)^{n/2}}\int_{\epsilon/\sigma}^{\sqrt{D(\mathbf{K})}/\sigma}t^n \exp\left(-{t^2\over 2}\right) dt,$$
and thus
\begin{eqnarray*}\label{eqphi2}
\varphi_2 \le {C_{\mathbf{K},a}n\gamma_n\sigma\over (2\pi)^{n/2}}\int_{0}^{+\infty}t^n \exp\left(-{t^2\over 2}\right) dt={C_{\mathbf{K},a}n\gamma_n\sigma \over (2\pi)^{n/2}}p(n).
\end{eqnarray*}
Here we have set  $p(n):=\int_{0}^{+\infty}t^n e^{-{t^2\over 2}}dt$ which can be checked to be given by  (\ref{constantpn}) (e.g., using  induction on $n$).
\smallskip\noindent
Now, combining with the upper bound for $C_{\mathbf{K},a}$ from (\ref{ineqcrad}), we obtain
\begin{eqnarray*}
\varphi_2\le {n\sigma^{n+1}p(n)\over \epsilon^n\eta_{\mathbf{K}}}e^{{\epsilon^2\over 2\sigma^2}}.
\end{eqnarray*}

\smallskip\noindent
Therefore, we have shown:
$$\varphi=\varphi_1+\varphi_2\le\epsilon+{n\sigma^{n+1}p(n)\over \epsilon^n\eta_{\mathbf{K}}}e^{{\epsilon^2\over 2\sigma^2}},$$
which shows the lemma. \qed
\end{proof}

\smallskip\noindent
We are now ready to prove Theorem \ref{thmfrk}.

\smallskip
\begin{proof}{\em (of Theorem \ref{thmfrk})}
Observe that, if $f$ is a polynomial, then we can use the upper bound (\ref{ineqmf}) for its Lipschitz constant and thus the inequality  (\ref{thmfrkeq2}) follows as a direct consequence of the inequality (\ref{thmfrkeq1}).
Therefore, it suffices to show the relation (\ref{thmfrkeq1}).

\smallskip
\noindent
Recall that $a$ is a minimizer of $f$ over $\mathbf{K}$. 
As $f$ is Lipschitz continuous with Lipschitz constant $M_f$ on $K$, we have
$$f(x)-f(a)\le M_f\|x-a\| \ \ \forall x\in \mathbf{K}.$$
This implies
$$f^{(r)}_{\mathbf{K},a}-f_{\min,\mathbf{K}}=\int_{\mathbf{K}}c^r_{\mathbf{K},a}{H_{r,a}}(x)(f(x)-f(a))dx \le M_f\int_{\mathbf{K}} \|x-a\| c^r_{\mathbf{K},a}H_{r,a}(x)dx.$$

\smallskip
\noindent
Our objective is now to show the existence of a constant $\zeta(\mathbf{K})$ such that
\begin{equation*}
\psi:=\int_{\mathbf{K}} c^r_{\mathbf{K},a}\|x-a\| H_{r,a}(x)dx\le {\zeta(\mathbf{K})\over \sqrt{2r+1}}, \ \ \text{for all $r\ge r_{\bK}$, (see (\ref{addefrK}))}
\end{equation*}
by which we can then conclude the proof for (\ref{thmfrkeq1}).

\smallskip\noindent
For this, we split the integral $\psi$ as the sum of two terms:

\begin{eqnarray*}
\psi&=&\underbrace{\int_{\mathbf{K}} c^r_{\mathbf{K},a}\|x-a\| G_a(x)dx}_{=:\psi_1}+\underbrace{\int_{\mathbf{K}} c^r_{\mathbf{K},a}\|x-a\| (H_{r,a}(x)-G_a(x))dx.}_{=:\psi_2}
\end{eqnarray*}

\smallskip
\noindent
First, we upper bound the term $\psi_1$. As $c^r_{\mathbf{K},a}\le C_{\mathbf{K},a}$ (by (\ref{ineqcrad})), we can use Lemma \ref{propga} to conclude that, for all $0<\epsilon\le \epsilon_{\mathbf{K}}$,
\begin{equation}\label{psi1}
\psi_1\le \int_{\mathbf{K}}C_{\mathbf{K},a}\|x-a\|G_a(x)dx \le \epsilon+{n\sigma^{n+1}p(n)\over \epsilon^n\eta_{\mathbf{K}}}e^{{\epsilon^2\over 2\sigma^2}}=\epsilon\underbrace{\left[1+{n\sigma^{n+1}p(n)\over \epsilon^{n+1}\eta_{\mathbf{K}}}e^{{\epsilon^2\over 2\sigma^2}}\right]}_{=:\mu_1}=\epsilon \mu_1.
\end{equation}

\vspace{0.5cm}\noindent
Second we bound the integral $$\psi_2=\int_{\mathbf{K}} c^r_{\mathbf{K},a}\|x-a\| (H_{r,a}(x)-G_a(x))dx.$$

\smallskip\noindent
We can upper bound the function $H_{r,a}(x)-G_a(x)$ using the estimate from (\ref{fmf2r}) and we get
\begin{eqnarray*}
H_{r,a}(x)-G_a(x)&\le& {{1\over (2\pi\sigma^2)^{n/2}}} {\|x-a\|^{4r+2}\over (2\sigma^2)^{2r+1}(2r+1)!}.
\end{eqnarray*}

\smallskip
\noindent
Then we have
$$\psi_2\le {{1\over (2\pi\sigma^2)^{n/2}}}\int_{\mathbf{K}}c^r_{\mathbf{K},a} {\|x-a\|^{4r+3}\over (2\sigma^2)^{2r+1}(2r+1)!}dx= {{1\over (2\pi\sigma^2)^{n/2}}}{c^r_{\mathbf{K},a}\over (2\sigma^2)^{2r+1}(2r+1)!}\int_{\mathbf{K}}\|x-a\|^{4r+3}dx.$$

\smallskip\noindent
Now we upper bound the integral $\int_{\mathbf{K}}\|x-a\|^{4r+3}dx$.
Since $\mathbf{K}\subseteq B_{\sqrt{D(\mathbf{K})}}(a)$, one has
\begin{equation*}
\int_{\mathbf{K}}\|x-a\|^{4r+3}dx \le \int_{B_{\sqrt{D(\mathbf{K})}}(a)} \|x-a\|^{4r+3}dx,
\end{equation*}
where the right hand side, by Lemma \ref{lemgf}, is equal to
$$n\gamma_n\int_{0}^{\sqrt{D(\mathbf{K})}}z^{4r+n+2} dz={n\gamma_nD(\mathbf{K})^{{4r+n+3\over 2}}\over 4r+n+3} \le n\gamma_nD(\mathbf{K})^{{4r+n+3\over 2}}.$$

\smallskip\noindent
Thus, we obtain
$$\psi_2\le {{1\over (2\pi\sigma^2)^{n/2}}}{c^r_{\mathbf{K},a}\over (2\sigma^2)^{2r+1}(2r+1)!}n\gamma_nD(\mathbf{K})^{{4r+n+3\over 2}}.$$

\smallskip\noindent
We now use the upper bound for $c^r_{\mathbf{K},a}$ from (\ref{ineqcrad}):
\begin{equation*}
c^r_{\mathbf{K},a}\le \frac{(2\pi\sigma^2)^{n/2}\exp\left({\epsilon^2\over 2\sigma^2}\right)}{\eta_{\mathbf{K}} \epsilon^n \gamma_n}
\end{equation*}
and we obtain
$$\psi_2\le {n\exp\left(\epsilon^2\over2\sigma^2\right)D(\mathbf{K})^{{4r+n+3\over 2}}\over \eta_{\mathbf{K}}\epsilon^n(2r+1)!(2\sigma^2)^{2r+1}}.$$

\smallskip
\noindent
Finally we use the Stirling's inequality:
$$(2r+1)!\ge \sqrt{2\pi(2r+1)}\left({2r+1}\over e\right)^{2r+1},$$
and obtain
\begin{eqnarray}\label{psi2}
\psi_2&\le& \underbrace{{n\exp\left(\epsilon^2\over2\sigma^2\right)D(\mathbf{K})^{{n+1\over 2}}\over \eta_{\mathbf{K}}}}_{=:\mu_2}\left( D(\mathbf{K})e\over 2\sigma^2{\epsilon^{n/(2r+1)}}(2r+1) \right)^{2r+1}{1\over \sqrt{2\pi(2r+1)}}\\
&=& {\mu_2\over\sqrt{2\pi(2r+1)}}\left( {D(\mathbf{K})e\over 2\sigma^2{\epsilon^{n/(2r+1)}}(2r+1)} \right)^{2r+1}.\nonumber
\end{eqnarray}

\smallskip
\noindent
We can now upper bound the quantity $\psi=\psi_1+\psi_2$, by combining the upper bound for $\psi_1$ in (\ref{psi1}) with the above upper bound (\ref{psi2}) for $\psi_2$. That is,

\begin{equation*}
\psi\le \epsilon\mu_1+{\mu_2\over\sqrt{2\pi(2r+1)}}\left( {D(\mathbf{K})e\over 2\sigma^2{\epsilon^{n/(2r+1)}}(2r+1)} \right)^{2r+1}.
\end{equation*}

\smallskip\noindent
We now indicate how to select the parameters  $\epsilon$ and $\sigma$.

\smallskip\noindent
First we select $\sigma=\epsilon$, so that both parameters $\mu_1$ and $\mu_2$ appearing in (\ref{psi1}) and (\ref{psi2}) are constants depending on $n$ and $\mathbf{K}$, namely

$$\mu_1=1+{np(n)e^{{1/2}}\over \eta_{\mathbf{K}}} \ \
\text{and}\ \ \mu_2={ne^{{1/2}}D(\mathbf{K})^{{n+1\over 2}}\over \eta_{\mathbf{K}}}.$$

\smallskip\noindent
Next we select $\epsilon$ so that ${D(\mathbf{K})e\over 2{\epsilon^{2+n/(2r+1)}}(2r+1)}=1$, i.e.,
$${ \epsilon=\left({D(\mathbf{K})e\over 2(2r+1)}\right)^{2r+1\over 2(2r+1)+n}=\left({D(\mathbf{K})e\over 2}\right)^{2r+1\over 2(2r+1)+n}\left({1\over 2r+1}\right)^{{1\over2}-{n\over 4(2r+1)+2n}} .}$$

\smallskip\noindent
Summarizing, we have shown that
{\begin{eqnarray}
\psi&\le& \left({1\over 2r+1}\right)^{{1\over2}-{n\over 4(2r+1)+2n}} \left[\left({D(\mathbf{K})e\over 2}\right)^{2r+1\over 2(2r+1)+n}\mu_1+{\mu_2\over \sqrt{2\pi}}\left({1\over 2r+1}\right)^{n\over 4(2r+1)+2n}\right]\nonumber\\
&\le& \left({1\over 2r+1}\right)^{{1\over2}} 6n\left(\mu_1\max\left\{ 1,\sqrt{D(\mathbf{K})e\over 2} \right\}+{\mu_2\over \sqrt{2\pi}}\right). \label{aadre1}
\end{eqnarray}}

\noindent To obtain the last inequality (\ref{aadre1}), we use the inequality $\left({1\over 2r+1}\right)^{-{n\over 4(2r+1)+2n}}< 6n$ (recall Lemma \ref{lemadd1}), together with the two inequalities $\left({D(\mathbf{K})e\over 2}\right)^{2r+1\over 2(2r+1)+n}\le \max\left\{ 1,\sqrt{D(\mathbf{K})e\over 2} \right\}$ and $\left({1\over 2r+1}\right)^{n\over 4(2r+1)+2n}\le1$.

\smallskip
\noindent
Since we have assumed $\epsilon\le \epsilon_{\mathbf{K}}$ (recall Lemma \ref{lemcra}), this implies the condition ${r\ge{D(\mathbf{K})e\over 4}{\epsilon_{\bK}^{-\left(2+{n\over 2r+1}\right)}}-{1\over 2}}$, i.e., the inequality (\ref{aadre1}) holds for all ${r\ge{D(\mathbf{K})e\over 4}{\epsilon_{\bK}^{-\left(2+{n\over 2r+1}\right)}}-{1\over 2}}$.
If $\epsilon_\bK\le 1$ and $r\ge n/2$, then we have
$\epsilon_\bK^{-(2+{n\over 2r+1})} \le \epsilon_{\bK}^{-3}$ and thus the inequality
(\ref{aadre1}) holds for all
${r\ge \max\left\{ {D(\mathbf{K})e\over 4\epsilon_{\bK}^{3}}, {n\over 2}\right\}}$.
If $\epsilon_\bK\ge 1$ then $\epsilon_\bK^{-(2+{n\over 2r+1})} \le 1$ and thus (\ref{aadre1})  holds for all integers $r\ge {D(\bK)e\over 4}$.
Hence, the inequality (\ref{aadre1}) holds for all $r\ge {r_{\bK}/ 2}$, where $r_\bK$ is as defined in (\ref{addefrK}).

\smallskip
\noindent
Finally, by defining the constant
{\begin{equation*}\label{eqzetak}
\zeta(\mathbf{K}):=6n\left(\mu_1\max\left\{ 1,\sqrt{D(\mathbf{K})e\over 2} \right\}+{\mu_2\over \sqrt{2\pi}}\right),
\end{equation*}}

\smallskip
\noindent
which indeed depends only on $\bK$ and its dimension $n$,
we can conclude  the proof for (\ref{thmfrkeq1}).
\qed
\end{proof}

\begin{remark}\label{5rkminasp}
Note that in the proof of Theorem \ref{thmfrk}, we  use Assumption \ref{propeta} only for the selected minimizer $a\in\mathbf{K}$ (and we use it only in the proof of Lemma \ref{lemcra}). Hence, if the selected point $a$ lies in the interior of $\mathbf{K}$, i.e.,  if there exists $\delta>0$ such that $B_\delta(a)\subseteq\mathbf{K}$, then the result of Theorem \ref{thmfrk} (and thus Theorem \ref{thmmain}) holds when selecting $\eta_{\mathbf{K}}=1$ and $\epsilon_{\mathbf{K}}=\delta$.

\smallskip
\noindent
Our results extend also to unconstrained global minimization:
$$f^*:=\min_{x\in\oR^n}f(x),$$

\smallskip\noindent
if we know that $f$ has a global minimizer $a$ and we know a ball $B_\delta(0)$ containing $a$. We can then indeed minimize $f$ over a compact set $K$, which can be chosen to be the ball $B_\delta(0)$ or a suitable hypercube containing $a$.
\end{remark}

\section{Obtaining feasible solutions through sampling}\label{secgenerate}

In this section we indicate how to sample feasible points in the set $\bK$ from the optimal density function obtained by solving the semidefinite program (\ref{fundr}). 

\smallskip\noindent
Let $f\in\oR[x]$ be a polynomial.
Suppose $h^*(x)\in\Sigma[x]_r$ is an optimal solution of the program (\ref{fundr}), i.e., $\underline{f}^{(r)}_{\mathbf{K}}=\int_{\mathbf{K}} f(x)h^*(x)dx$ and $\int_{\bK}h^*(x)dx=1$.
\smallskip\noindent
Then $h^*$ can be seen as the probability density function of a probability distribution on $\bK$, denoted as $\mathcal{T}_{\mathbf{K}}$ and,
for all random vector $X=(X_1,\ldots,X_n)\sim\mathcal{T}_{\mathbf{K}}$, the expectation of $f(X)$ is given by:
\begin{equation}\label{expectation}
\mathbb{E}\left[f(X)\right]=\int_{\mathbf{K}} f(x)h^*(x)dx=\underline{f}^{(r)}_{\mathbf{K}}.
\end{equation}

\smallskip\noindent
As we now recall one can generate  random samples $x\in \bK$  from the distribution $\mathcal{T}_{\mathbf{K}}$ using the well known {\em method of conditional distributions} (see e.g., \cite[Section 8.5.1]{law07}). Then we will observe that with high probability one of these sample points satisfies (roughly) the inequality $f(x)\le\underline{f}^{(r)}_{\mathbf{K}}$ (see Theorem \ref{thmpb} for details). 

\smallskip
\noindent
In order to sample a random vector $X=(X_1,\ldots,X_n) \sim \mathcal{T}_{\mathbf{K}}$,
we assume that, for each $i=2,\ldots,n$, we know the cumulative conditional distribution of $X_i$ given that $X_j=x_j$ for $j=1,\ldots,i-1$,  defined in terms of
probabilities as
\[
F_i(x_i \mid x_1,\ldots,x_{i-1}) := \mathbf{Pr}\left[X_i \le x_i \; \mid \; X_1 = x_1,\ldots, X_{i-1}= x_{i-1}\right].
\]
Additionally, we assume that we know the  cumulative marginal distribution function of $X_i$, defined as:
\[
F_i(x_i)   := \mathbf{Pr}\left[X_i \le x_i\right].
\]
 Then one can generate a  random sample
$x=(x_1,\ldots,x_n) \in \bK$ from the distribution $\mathcal{T}_{\mathbf{K}}$ by the following algorithm:


\begin{description}
\item[$\bullet$] Generate $x_1$ with cumulative distribution function $F_1(\cdot)$.
\smallskip 
\item[$\bullet$] Generate $x_2$ with cumulative distribution function $F_2\left(\cdot|x_1\right).$ \\
\vdots
\smallskip
\item[$\bullet$] Generate $x_n$ with cumulative distribution function $F_n\left(\cdot|x_1,\dots,x_{n-1}\right).$
\end{description}
Then return $x=(x_1,x_2,\dots,x_n)^T$.

\smallskip\noindent
There remains to explain how to generate a (univariate) sample point $x$ with a given cumulative distribution function $F(\cdot)$, since this operation is carried out at each of the $n$ steps of the above algorithm.
For this  one can use the classical  {\em inverse-transform method} (see e.g., \cite[Section 8.2.1]{law07}), which reduces to sampling from the uniform distribution on $[0,1]$ and can be  described as follows:

\begin{description}
\item[$\bullet$] Generate a sample $u$ from the uniform distribution over $[0,1]$.
\smallskip\noindent
\item[$\bullet$] Return $x=F^{-1}(u)$ (if $F$ is strictly monotone increasing, or $x=\min\{y: F(y)\ge u\}$ otherwise).
\end{description}
{
Hence, in order to be able to apply the method of conditional distributions for sampling from $\mathbf K$ we need to solve the equation $x=F^{-1}(u)$.
For instance, when $F(\cdot)$ is a univariate polynomial, solving the equation $x=F^{-1}(u)$ reduces to computing the eigenvalues of the corresponding companion matrix (see, e.g., \cite[Section 2.4.1]{ML09}). This applies, e.g., when $\mathbf K$ is the hypercube or the simplex, as we see below. }

\bigskip\noindent
As an illustration, we first indicate how to compute the cumulative marginal and conditional distributions $F_i(\cdot)$ and $F_i(\cdot\mid x_1\ldots x_{i-1})$ for the case of the hypercube $\mathbf K=\oQ_n=[0,1]^n$.
As before we are given a sum of squares density function $h^*(x)$ on  $[0,1]^n$.
 For $i=1,\ldots,n$, define the  polynomial function $f_{1\ldots i} \in \oR[x_1,\ldots,x_i]$ by
\begin{equation}\label{relfi}
f_{1 \ldots i}(x_1,\ldots,x_i)=\int_0^1 \cdots \int_0^1 h^*(x_1,\ldots,x_n) dx_{i+1}\cdots dx_n.
\end{equation}
Then the cumulative marginal distribution function $F_1(\cdot)$ is given by
$$F_1(x_1)= \int_0^{x_1} f_1(y)dy$$
and, for $i=2,\ldots,n$, the  cumulative conditional distribution function $F_i(\cdot\mid x_1\ldots x_{i-1})$ is given by
$$F_i(x_i\mid x_1\ldots x_{i-1}) ={ \int_0^{x_i} f_{1\ldots i}(x_1,\ldots,x_{i-1},y) dy \over f_{1\ldots (i-1)}(x_1,\ldots,x_{i-1})}.$$
{
The computation of the cumulative marginal and conditional distributions can be carried out in the same way  for the simplex  $\mathbf K=\Delta_n$, after replacing the function $f_{1\ldots i} \in \oR[x_1,\ldots,x_i]$
in  (\ref{relfi}) by
$$f_{1 \ldots i}(x_1,\ldots,x_i)=\int_0^{1-x_{i}-x_{i+1}-\cdots-x_{n-1}}\int_0^{1-x_{i}-\cdots-x_{n-2}} \cdots \int_0^{1-x_i} h^*(x_1,\ldots,x_n) dx_{i+1}\cdots dx_n.$$
Note that in both cases the functions $F_i(x_i\mid x_1\ldots x_{i-1})$ are indeed univariate polynomials. We will  apply this sampling method to several examples of  polynomial minimization over the hypercube and the simplex  in the next section.}

\medskip\noindent
We now  observe that if we generate sufficiently many samples from the distribution $\mathcal{T}_{\mathbf{K}}$ then, with high probability, one of these samples is a point  $x\in\mathbf{K}$ satisfying (roughly)  $f(x)\le\underline{f}^{(r)}_{\mathbf{K}}$.

%

\begin{theorem}\label{thmpb}
Let $X \sim \mathcal{T}_{\mathbf{K}}$.
For all $\epsilon>0$,
\[
 \mathbf{Pr}\left[
 f(X) \ge  \underline{f}^{(r)}_{\mathbf{K}}+\epsilon\left(\underline{f}_{\mathbf{K}}^{(r)}-f_{\min,\mathbf{K}}\right)\right] \le \frac{1}{1+\epsilon}.
 \]
\end{theorem}

\begin{proof}

Let  $X\sim\mathcal{T}_{\mathbf{K}}$ so that $\mathbb{E}\left[f(X)\right]=\underline{f}^{(r)}_{\mathbf{K}}$.
Define the nonnegative random variable
\begin{equation*}\label{eqyx}
Y:=f(X)-f_{\min,\mathbf{K}}.
\end{equation*}
Then, one has $\mathbb{E}\left[Y\right]=\underline{f}^{(r)}_{\mathbf{K}}-f_{\min,\mathbf{K}}$.
Given  $\epsilon>0$,
 the Markov Inequality (see e.g., \cite[Theorem 3.2]{MR95}) implies
\begin{equation*}
\mathbf{Pr}\left[ Y \ge (1+\epsilon)\mathbb{E}\left[Y\right]\right]\le {1\over 1+\epsilon}.
\end{equation*}
This completes the proof.
\qed
\end{proof}
For given $\epsilon >0$, if one samples $N$ times independently from  $\mathcal{T}_{\mathbf{K}}$, one therefore obtains
an $x \in \bK$ such that
\[
f(x) <  \underline{f}^{(r)}_{\mathbf{K}}+\epsilon\left(\underline{f}_{\mathbf{K}}^{(r)}-f_{\min,\mathbf{K}}\right)
\]
with probability at least $1 - \left(\frac{1}{1+\epsilon}\right)^N$. For example, if $N \ge 1 + \frac{1}{\epsilon}$ then this probability
is at least $1-1/e$.

\section{Numerical examples}
\label{sec:numerial examples}
In this section, we consider several well-known polynomial test functions from global optimization that are listed in Table \ref{table:example}.

\begin{table}[h!]
\caption{Test functions}
\label{table:example}
  \begin{tabular}{| m{3cm} | m{4cm} |m{4cm} | m{2cm} |}
    \hline
    Name & Formula & Minimum ($f_{\min,\mathbf{K}}$) & Search domain ($\mathbf{K}$) \\ \hline
    Booth Function & $f=(x_1+2x_2-7)^2+(2x_1+x_2-5)^2$ & $f(1,3)=0$ & $[-10,10]^2$ \\ \hline
    Matyas Function & $f=0.26(x_1^2+x_2^2)-0.48x_1x_2$ & $f(0,0)=0$ & $[-10,10]^2$ \\ \hline
    Three--Hump Camel Function & $f=2x_1^2-1.05x_1^4+{1\over6}x_1^6+x_1x_2+x_2^2$ & $f(0,0)=0$ & $[-5,5]^2$ \\ \hline
    Motzkin Polynomial & $f=x_1^4x_2^2+x_1^2x_2^4-3x_1^2x_2^2+1$ & $f(\pm1,\pm1)=0$ & $[-2,2]^2$ \\ \hline
    Styblinski--Tang Function ($n$-variate) & $f=\sum_{i=1}^n{1\over 2}x_i^4-8x_i^2+{5\over 2}x_i  $ & $f(-2.093534, \dots ,-2.093534)=-39.16599n$ & $[-5,5]^n$ \\ \hline
    Rosenbrock Function ($n$-variate) & $f=\sum_{i=1}^{n-1} 100(x_{i+1}-x_i^2)^2+(x_i-1)^2 $ & $f(1,\dots,1)=0$ & $[-2.048,2.048]^n$ \\
    \hline
    Matyas Function (Modified-S)& $f=0.26[(20x_1-10)^2+(20x_2-10)^2]-0.48(20x_1-10)(20x_2-10)$ &$f(0.5,0.5)=0$ & $\Delta_2$ \\ \hline

   Three-Hump Camel Function (Modified-S)& $f=2(10x_1-5)^2-1.05(10x_1-5)^4+{1\over 6}(10x_1-5)^6+(10x_1-5)(10x_2-5)+(10x_2-5)^2$ &$f(0.5,0.5)=0$ & $\Delta_2$\\
    \hline
   Matyas Function (Modified-B) & $f=0.26[(20x_1^2-10)^2+(20x_2^2-10)^2]-0.48(20x_1^2-10)(20x_2^2-10)$ &$f(\pm{\sqrt{2}\over 2},\pm{\sqrt{2}\over 2})=0$ & $B_1(0)$ \\ \hline

   Three-Hump Camel Function (Modified-B) & $f=2(10x_1^2-5)^2-1.05(10x_1^2-5)^4+{1\over 6}(10x_1^2-5)^6+(10x_1^2-5)(10x_2^2-5)+(10x_2^2-5)^2$ &$f(\pm{\sqrt{2}\over 2},\pm{\sqrt{2}\over 2})=0$ & $B_1(0)$\\ \hline

    \hline
  \end{tabular}
  \end{table}

\noindent For these functions, we calculate the parameter $\underline{f}_{\mathbf{K}}^{(r)}$ by solving the SDP  (\ref{fundr2}) for increasing values of the order $r$. As already mentioned by Lasserre \cite[Section 4]{Las11}, this computation may be done as a
generalised eigenvalue problem --- one does not actually have to use an SDP solver. This follows from the fact that the SDP  (\ref{fundr2})
only has one constraint.
In particular,
  $\underline{f}_{\mathbf{K}}^{(r)}$ {is equal to the largest scalar $\lambda$ for which $A-\lambda B\succeq 0$, i.e.,  the smallest  generalized eigenvalue} of the system:
\[
Ax = \lambda Bx \quad \quad\quad (x \neq 0),
\]
where the symmetric matrices $A$ and $B$ are of order ${n + r \choose r}$ with rows and columns  indexed by $N(n,r)$,
and
\begin{equation}
\label{matrices A and B}
A_{\alpha, \beta} = \sum_{\delta \in N(n,d)} f_\delta \int_{\mathbf{K}} x^{\alpha + \beta + \delta} dx, \quad B_{\alpha, \beta} = \int_{\mathbf{K}} x^{\alpha + \beta} dx \quad \alpha, \beta \in {N}(n,r).
\end{equation}
We performed the computation on a PC with Intel(R) Core(TM) i7-4600U CPU (2.10 GHz) and with 8 GB RAM. The generalized eigenvalue computation was done in Matlab using the {\tt eig} function.

%
%

\smallskip\noindent
We record the values $\underline{f}_{\mathbf{K}}^{(r)}$ as well as the CPU times in
Tables \ref{table:result1}, \ref{table:n10}, \ref{table:n15}, \ref{table:n20},  and \ref{tab:overview} for minimization over the hypercube, the simplex and the ball. Note that we only list the time for solving the generalised eigenvalue problem, and not for constructing the matrices $A$ and $B$ in  (\ref{matrices A and B}). In other words, we  assume the necessary moments are computed beforehand, and that the time needed to construct the {matrices $A$ and $B$} in  (\ref{matrices A and B}) is negligible if the relevant moments are known.

For instance, in Table \ref{table:result1}, we have $n=2$ and we can compute the parameter $\underline{f}_{\mathbf{K}}^{(r)}$  up to order $r=20$ for  four test functions. Moreover, in Tables \ref{table:n10}, \ref{table:n15} and \ref{table:n20}, we have $n=10,15,20$, respectively, and the parameter $\underline{f}_{\mathbf{K}}^{(r)}$ can be computed up to order $r=5$, $r=4$ and $r=3$, respectively. Note that in all cases the computation  is very fast (at most a few seconds). However, for larger values of $n$ or $r$ we sometimes encountered numerical instability. This may be due to inaccurate calculation of the moments, or to  inherent ill-conditioning of the
matrices $A$ and $B$ in  (\ref{matrices A and B}). These issues are of practical importance, but beyond the scope of the present study. Also,
one must bear in mind that the order of the matrices $A$ and $B$ grows as ${n+r \choose r}$, and this imposes  a practical limit on how large the values of $n$ and $r$ may be when computing $\underline{f}_{\mathbf{K}}^{(r)}$.

\begin{table}[h!]
\caption{$\underline{f}_{\mathbf{K}}^{(r)}$ for Booth, Matyas, Three--Hump Camel and Motzkin Functions over the hypercube}
\label{table:result1}
  \begin{tabular}{| c | m{1.3cm} | m{1cm} | m{1.3cm} | m{1cm} | m{1.3cm} | m{1cm} | m{1.4cm} | m{1cm} |}
    \hline
\multirow{2}{*}{r} & \multicolumn{2}{c|}{Booth Function} & \multicolumn{2}{c|}{Matyas Function} & \multicolumn{2}{m{3cm}|}{Three--Hump Camel Function}& \multicolumn{2}{c|}{Motzkin Polynomial} \\ \cline{2-9}
                   & Value  & Time (sec.)& Value  & Time (sec.)  & Value  & Time (sec.)  & Value  & Time (sec.)\\ \hline
$1$  & $244.680$ & $0.000666$ & $8.26667$ & $0.000739$ & $265.774$ & $0.000742$     & $4.2$      & $0.000719$\\ \hline
$2$  & $162.486$ & $0.000061$ & $5.32223$ & $0.000072$ & $29.0005$ & $0.000062$     & $1.06147$  & $0.000088$\\ \hline
$3$  & $118.383$ & $0.000083$ & $4.28172$ & $0.000072$ & $29.0005$ & $0.000066$     & $1.06147$  & $0.000080$\\ \hline
$4$  & $97.6473$ & $0.000079$ & $3.89427$ & $0.000119$ & $9.58064$ & $0.000117$     & $0.829415$ & $0.000118$\\ \hline
$5$  & $69.8174$ & $0.000171$ & $3.68942$ & $0.000208$ & $9.58064$ & $0.000177$     & $0.801069$ & $0.000189$\\ \hline
$6$  & $63.5454$ & $0.000277$ & $2.99563$ & $0.000263$ & $4.43983$ & $0.000263$     & $0.801069$ & $0.000208$\\ \hline
$7$  & $47.0467$ & $0.000423$ & $2.54698$ & $0.000343$ & $4.43983$ & $0.001146$     & $0.708889$ & $0.000395$\\ \hline
$8$  & $41.6727$ & $0.000587$ & $2.04307$ & $0.000417$ & $2.55032$ & $0.000647$     & $0.565553$ & $0.000584$\\ \hline
$9$  & $34.2140$ & $0.000657$ & $1.83356$ & $0.000655$ & $2.55032$ & $0.000586$     & $0.565553$ & $0.000766$\\ \hline
$10$ & $28.7248$ & $0.000997$ & $1.47840$ & $0.000780$ & $1.71275$ & $0.000782$     & $0.507829$ & $0.001210$\\ \hline
$11$ & $25.6050$ & $0.001181$ & $1.37644$ & $0.009241$ & $1.71275$ & $0.001026$     & $0.406076$ & $0.001261$\\ \hline
$12$ & $21.1869$ & $0.001942$ & $1.11785$ & $0.001753$ & $1.2775$  & $0.001693$     & $0.406076$ & $0.001712$\\ \hline
$13$ & $19.5588$ & $0.002352$ & $1.0686$  & $0.001857$ & $1.2775$  & $0.002031$     & $0.3759$   & $0.003427$\\ \hline
$14$ & $16.5854$ & $0.002829$ & $0.8742$  & $0.002253$ & $1.0185$  & $0.002629$     & $0.3004$   & $0.003711$\\ \hline
$15$ & $15.2815$ & $0.003618$ & $0.8524$  & $0.002270$ & $1.0185$  & $0.002936$     & $0.3004$   & $0.002351$\\ \hline
$16$ & $13.4626$ & $0.003452$ & $0.7020$  & $0.003580$ & $0.8434$  & $0.003452$     & $0.2819$   & $0.003672$\\ \hline
$17$ & $12.2075$ & $0.004248$ & $0.6952$  & $0.004662$ & $0.8434$  & $0.004652$     & $0.2300$   & $0.004349$\\ \hline
$18$ & $11.0959$ & $0.005217$ & $0.5760$  & $0.005510$ & $0.7113$  & $0.004882$     & $0.2300$   & $0.006060$\\ \hline
$19$ & $9.9938$  & $0.007200$ & $0.5760$  & $0.005610$ & $0.7113$  & $0.006752$     & $0.2185$   & $0.007641$\\ \hline
$20$ & $9.2373$  & $0.009707$ & $0.4815$  & $0.006975$ & $0.6064$  & $0.007031$     & $0.1817$   & $0.007686$\\ \hline
  \end{tabular}
  \end{table}

\begin{table}[h!]
\caption{$\underline{f}_{\mathbf{K}}^{(r)}$ for Styblinski--Tang and Rosenbrock Functions (with $n=10$) over the hypercube}\label{table:n10}
\begin{tabular}{| c | c | m{1cm} | c | m{1cm} |}
\hline
\multirow{2}{*}{r}&\multicolumn{2}{c|}{Sty.--Tang ($n=10$)}&\multicolumn{2}{c|}{Rosenb. ($n=10$)}\\ \cline{2-5}
           & Value  & Time (sec.)    & Value  & Time (sec.)   \\ \hline
    $1$   & $-57.1688$  & $0.098$      & $3649.85$  & $0.0005$  \\  \hline
    $2$   & $-94.5572$  & $0.001$      & $2813.66$  & $0.0009$  \\  \hline
    $3$   & $-108.873$  & $0.011$      & $2393.63$  & $0.0156$  \\  \hline
    $4$   & $-132.8810$ & $0.349$     & $1956.81$  & $0.4004$  \\ \hline
    $5$   & $-146.7906$ & $9.245$     & $1701.85$  & $12.997$     \\ \hline
  \end{tabular}
  \end{table}

 \begin{table}[h!]
\caption{$\underline{f}_{\mathbf{K}}^{(r)}$ for Styblinski--Tang and Rosenbrock Functions (with $n=15$) over the hypercube}\label{table:n15}
\begin{tabular}{| c | c | c |c|c|}
\hline
\multirow{2}{*}{r} & \multicolumn{2}{c|}{Sty.--Tang ($n=15$)}&\multicolumn{2}{c|}{Rosenb. ($n=15$)}\\ \cline{2-5}
           & Value  & Time (sec.)    & Value  & Time (sec.) \\ \hline
    $1$   & $-82.8311$& $0.001071$  & $5887.5$ & $0.094693$     \\ \hline
    $2$   & $-130.464$& $0.001707$  & $4770.71$ & $0.002282$   \\ \hline
    $3$   & $-148.5594$ & $0.170907$  & $4160.78$ & $0.157897$  \\ \hline
    $4$   & $-180.9728$ & $16.796383$  & $3552.04$ & $24.696591$  \\ \hline
  \end{tabular}
  \end{table}

  \begin{table}[h!]
\caption{$\underline{f}_{\mathbf{K}}^{(r)}$ for Styblinski--Tang and Rosenbrock Functions (with $n=20$) over the hypercube}\label{table:n20}
\begin{tabular}{| c | c | c |c|c|}
\hline
\multirow{2}{*}{r} & \multicolumn{2}{c|}{Sty.--Tang ($n=20$)}&\multicolumn{2}{c|}{Rosenb. ($n=20$)}\\ \cline{2-5}
           & Value  & Time (sec.)    & Value  & Time (sec.) \\ \hline
    $1$   & $-107.875$ & $0.972741$      & $8158.36$  & $0.000949$  \\ \hline
    $2$   & $-164.11$ & $0.344403$      & $6806.74$  & $0.011370$   \\ \hline
    $3$   & $-185.6488$ & $2.655447$      & $6029.02$  & $2.955319$   \\ \hline
  \end{tabular}
  \end{table}


\begin{table}[h!]
\caption{$\underline{f}_{\mathbf{K}}^{(r)}$ for Matyas and Three-Hump Camel Functions (Modified) over the Simplex and the Euclidean ball. \label{tab:overview}}
  \begin{tabular}{|c| c | c | c | c|c|c|c|c|}\hline
  \multirow{2}{*}{r} & \multicolumn{2}{c|}{Matyas (Modified-S)}&\multicolumn{2}{c|}{Th.-H. C. (Modified-S)}& \multicolumn{2}{c|}{Matyas (Modified-B)}&\multicolumn{2}{c|}{Th.-H. C. (Modified-B)}\\ \cline{2-9}
           & Value  & Time (sec.)    & Value  & Time (sec.) & Value  & Time (sec.)    & Value  & Time (sec.) \\ \hline
    $1$   & $7.2243$ & $0.222604$ & $84.354$ & $0.000457$ & $18.000$ & $0.000379$      & $146.41$  & $0.000454$  \\ \hline
    $2$   & $4.6536$ & $0.000085$ & $22.398$ & $0.000081$ & $6.3995$ & $0.000049$      & $138.91$  & $0.000052$   \\ \hline
    $3$   & $3.9404$ & $0.000124$ & $12.353$ & $0.000115$ & $6.3995$ & $0.000054$      & $48.508$  & $0.000069$  \\ \hline
    $4$   & $3.7067$ & $0.000176$ & $3.9153$ & $0.000112$ & $4.4091$ & $0.000133$      & $39.673$  & $0.000111$   \\ \hline
    $5$   & $3.2317$ & $0.000696$ & $2.9782$ & $0.000489$ & $4.4091$ & $0.000187$      & $18.045$  & $0.000264$  \\ \hline
    $6$   & $2.7328$ & $0.000275$ & $1.3303$ & $0.000255$ & $3.9652$ & $0.000292$      & $13.881$  & $0.000309$   \\ \hline
    $7$   & $2.2985$ & $0.000511$ & $1.1773$ & $0.000334$ & $3.9652$ & $0.000323$      & $7.7876$  & $0.000300$  \\ \hline
    $8$   & $1.9536$ & $0.001432$ & $0.77992$& $0.000560$ & $3.8536$ & $0.000395$      & $5.7685$  & $0.000608$   \\ \hline
    $9$   & $1.6639$ & $0.000709$ & $0.73202$& $0.000666$ & $3.8536$ & $0.000517$      & $3.8699$  & $0.000636$  \\ \hline
   $10$   & $1.4293$ & $0.003370$ & $0.60846$& $0.001034$ & $3.4943$ & $0.000687$      & $2.8359$  & $0.000704$   \\ \hline
  \end{tabular}
  \end{table}


\smallskip\noindent
{Furthermore, 
we use the method described in Section \ref{secgenerate} to generate samples that are feasible solutions of (\ref{fundr}).
We report results  for the bivariate Rosenbrock and the Three--Hump Camel functions over the hypercube, and for the Matyas and Three-Hump Camel functions (Modified-S) over the simplex. 
 For each order $r\ge 1$,  the sample sizes $20$ and  $1000$ are used. We also generate samples uniformly from the feasible set, for comparison.
We give the results in Tables \ref{table:samplera}, \ref{table:samplerb}, \ref{table:samplerc} and \ref{table:samplerd},
where we record the mean, variance and the minimum value of these samples together with $\underline{f}_{\mathbf{K}}^{(r)}$ (which equals the sample mean by (\ref{expectation})).}

\begin{center}
  \begin{table}
\caption{Sampling results for the Rosenbrock Function ($n=2$) over the hypercube}
\label{table:samplera}
  \begin{tabular}{| c | c | c | c | c | c |}
\hline
     r                     & $\underline{f}_{\mathbf{K}}^{(r)}$ & Mean          & Variance     & Minimum    & Sample Size\\ \hline
    \multirow{2}{*}{$1$}   & \multirow{2}{*}{$214.648$}         & $121.125$     & $14005.5$  & $0.00451826$ & $20$ \\ \cline{3-6}
                           &                                    & $209.9$       & $80699.0$  & $0.0008754$  & $1000$\\ \hline
    \multirow{2}{*}{$2$}   & \multirow{2}{*}{$152.310$}         & $184.496$     & $58423.9$  & $4.94265$    & $20$\\ \cline{3-6}
                           &                                    & $149.6$       & $54455.0$  & $0.02805$    & $1000$\\ \hline
    \multirow{2}{*}{$3$}   & \multirow{2}{*}{$104.889$}         & $146.618$     & $64611.2$  & $0.0113339$  & $20$ \\ \cline{3-6}
                           &                                    & $110.1$       & $26022.0$  & $0.0665$     & $1000$\\ \hline
    \multirow{2}{*}{$4$}   & \multirow{2}{*}{$75.6010$}         & $62.4961$     & $5803.21$  & $0.0542813$  & $20$  \\ \cline{3-6}
                           &                                    & $75.65$       & $45777.0$  & $0.007285$   & $1000$\\ \hline
    \multirow{2}{*}{$5$}   & \multirow{2}{*}{$51.5037$}         & $58.4032$     & $4397.0$   & $0.668679$   & $20$ \\ \cline{3-6}
                           &                                    & $50.64$       & $6285.0$   & $0.01382$    & $1000$\\ \hline
    \multirow{2}{*}{$6$}   & \multirow{2}{*}{$41.7878$}         & $35.4183$     & $2936.24$  & $1.16154$    & $20$\\ \cline{3-6}
                           &                                    & $37.64$       & $3097.0$   & $0.06188$    & $1000$\\ \hline
    \multirow{2}{*}{$7$}   & \multirow{2}{*}{$30.1392$}         & $29.6545$     & $1022.2$   & $1.05813$    & $20$ \\ \cline{3-6}
                           &                                    & $27.11$       & $1332.0$   & $0.02044$    & $1000$\\ \hline
    \multirow{2}{*}{$8$}   & \multirow{2}{*}{$25.8329$}         & $19.5392$     & $301.334$  & $0.505628$   & $20$\\ \cline{3-6}
                           &                                    & $34.32$       & $4106.0$   & $0.074$      & $1000$\\ \hline
    \multirow{2}{*}{$9$}   & \multirow{2}{*}{$19.4972$}         & $20.8982$     & $328.475$  & $0.564992$   & $20$ \\ \cline{3-6}
                           &                                    & $18.65$       & $593.6$    & $0.07951$    & $1000$\\ \hline
    \multirow{2}{*}{$10$}  & \multirow{2}{*}{$17.3999$}         & $9.37959$     & $146.496$  & $0.562473$   & $20$   \\ \cline{3-6}
                           &                                    & $15.33$       & $685.7$    & $0.1448$     & $1000$\\ \hline
    \multirow{2}{*}{$11$}  & \multirow{2}{*}{$13.6289$}         & $8.74923$     & $52.1436$  & $0.75774$    & $20$ \\ \cline{3-6}
                           &                                    & $15.7$        & $7498.0$   & $0.1719$     & $1000$\\ \hline
    \multirow{2}{*}{$12$}  & \multirow{2}{*}{$12.5024$}         & $5.43151$     & $66.561$   & $0.438172$   & $20$  \\ \cline{3-6}
                           &                                    & $12.7$        & $764.7$    & $0.0945$     & $1000$\\ \hline
 \multicolumn{2}{|c|}{\multirow{2}{*}{Uniform Sample} }         & $489.722$     & $433549.0$ & $9.0754$     & $20$\\ \cline{3-6}
 \multicolumn{1}{|c}{}     &                                    & $465.729$     & $361150.0$ & $0.0771463$  & $1000$\\ \hline
  \end{tabular}
  \end{table}
\end{center}


\begin{center}
  \begin{table}
\caption{Sampling results for the Three--Hump Camel Function over the hypercube}
\label{table:samplerb}
  \begin{tabular}{| c | c | c | c | c | c |}
\hline
     r                     & $\underline{f}_{\mathbf{K}}^{(r)}$ & Mean          & Variance     & Minimum    & Sample Size\\ \hline
    \multirow{2}{*}{$1$}   & \multirow{2}{*}{$265.774$}         & $216.773$     & $177142.0$  & $0.106854$ & $20$ \\ \cline{3-6}
                           &                                    & $261.23$       & $193466.0$  & $0.11705$  & $1000$\\ \hline
    \multirow{2}{*}{$2$}   & \multirow{2}{*}{$29.0005$}         & $28.0344$     & $2964.85$  & $1.1718$    & $20$\\ \cline{3-6}
                           &                                    & $27.712$       & $6712.8$  & $0.014255$    & $1000$\\ \hline
    \multirow{2}{*}{$3$}   & \multirow{2}{*}{$29.0005$}         & $14.9951$     & $523.904$  & $0.452655$  & $20$ \\ \cline{3-6}
                           &                                    & $32.363$       & $16681.0$  & $0.0088426$     & $1000$\\ \hline
    \multirow{2}{*}{$4$}   & \multirow{2}{*}{$9.58064$}         & $2.99756$     & $14.1201$  & $0.175016$  & $20$  \\ \cline{3-6}
                           &                                    & $10.364$       & $1944.0$  & $0.010013$   & $1000$\\ \hline
    \multirow{2}{*}{$5$}   & \multirow{2}{*}{$9.58064$}         & $4.41907$     & $14.1358$   & $0.419394$   & $20$ \\ \cline{3-6}
                           &                                    & $9.1658$       & $643.88$   & $0.0015924$    & $1000$\\ \hline
    \multirow{2}{*}{$6$}   & \multirow{2}{*}{$4.43983$}         & $7.98481$     & $245.089$  & $0.126147$    & $20$\\ \cline{3-6}
                           &                                    & $4.5791$       & $493.12$   & $0.0035581$    & $1000$\\ \hline
    \multirow{2}{*}{$7$}   & \multirow{2}{*}{$4.43983$}         & $3.96711$     & $20.3193$   & $0.260331$    & $20$ \\ \cline{3-6}
                           &                                    & $3.7911$       & $57.847$   & $0.0076111$    & $1000$\\ \hline
    \multirow{2}{*}{$8$}   & \multirow{2}{*}{$2.55032$}         & $2.18925$     & $3.87943$  & $0.0310113$   & $20$\\ \cline{3-6}
                           &                                    & $2.2302$       & $8.3767$   & $0.0028817$      & $1000$\\ \hline
    \multirow{2}{*}{$9$}   & \multirow{2}{*}{$2.55032$}         & $1.38102$     & $2.27433$  & $0.138641$   & $20$ \\ \cline{3-6}
                           &                                    & $3.2217$       & $812.18$    & $0.00014805$    & $1000$\\ \hline
    \multirow{2}{*}{$10$}  & \multirow{2}{*}{$1.71275$}         & $1.03179$     & $0.992636$  & $0.0645815$   & $20$   \\ \cline{3-6}
                           &                                    & $1.5069$       & $3.9581$    & $0.0014225$     & $1000$\\ \hline
    \multirow{2}{*}{$11$}  & \multirow{2}{*}{$1.71275$}         & $1.30757$     & $1.90985$  & $0.0320489$    & $20$ \\ \cline{3-6}
                           &                                    & $1.6379$        & $7.2518$   & $0.0021144$     & $1000$\\ \hline
    \multirow{2}{*}{$12$}  & \multirow{2}{*}{$1.27749$}         & $0.841194$    & $0.914514$ & $0.0369565$   & $20$  \\ \cline{3-6}
                           &                                    & $1.2105$        & $2.3$    & $0.0005154$     & $1000$\\ \hline
 \multicolumn{2}{|c|}{\multirow{2}{*}{Uniform Sample} }         & $304.032$     & $163021.0$ & $1.65885$     & $20$\\ \cline{3-6}
 \multicolumn{1}{|c}{}     &                                    & $243.216$     & $183724.0$ & $0.00975034$  & $1000$\\ \hline
  \end{tabular}
  \end{table}
\end{center}


  \begin{table}
\caption{Sampling results for the Matyas Function (Modified-S) over the simplex}
\label{table:samplerc}
  \begin{tabular}{| c  |  c | c | c |c|c|}\hline
     r                     & $\underline{f}_{\mathbf{K}}^{(r)}$ & Mean          & Variance     & Minimum    & Sample Size\\ \hline
    \multirow{2}{*}{$1$}   & \multirow{2}{*}{$7.2243$}         & $6.3018$     & $37.373$  & $1.2448$ & $20$ \\ \cline{3-6}
                           &                                   & $7.0542$       & $64.863$  & $0.31812$  & $1000$\\ \hline
    \multirow{2}{*}{$2$}   & \multirow{2}{*}{$4.6536$}          & $5.7252$     & $34.964$  & $1.8924$    & $20$\\ \cline{3-6}
                           &                                   & $4.5932$       & $8.293$  & $0.91671$    & $1000$\\ \hline
    \multirow{2}{*}{$3$}   & \multirow{2}{*}{$3.9404$}          & $3.5187$     & $0.31411$  & $2.4465$  & $20$ \\ \cline{3-6}
                           &                                   & $3.7544$       & $1.3576$  & $0.071075$     & $1000$\\ \hline
    \multirow{2}{*}{$4$}   & \multirow{2}{*}{$3.7067$}          & $3.4279$     & $1.7187$  & $0.92913$  & $20$  \\ \cline{3-6}
                           &                                   & $3.8679$       & $6.5113$  & $0.027508$   & $1000$\\ \hline
    \multirow{2}{*}{$5$}   & \multirow{2}{*}{$3.2317$}          & $3.8273$     & $10.173$   & $0.40131$   & $20$ \\ \cline{3-6}
                           &                                    & $3.1485$       & $6.1263$   & $0.035796$    & $1000$\\ \hline
    \multirow{2}{*}{$6$}   & \multirow{2}{*}{$2.7328$}          & $2.2606$     & $3.3343$  & $0.2595$    & $20$\\ \cline{3-6}
                           &                                  & $2.5997$       & $10.8$   & $0.0016761$    & $1000$\\ \hline
    \multirow{2}{*}{$7$}   & \multirow{2}{*}{$2.2985$}         & $2.4568$     & $4.1652$   & $0.18947$    & $20$ \\ \cline{3-6}
                           &                                   & $2.1541$       & $12.868$   & $0.002669$    & $1000$\\ \hline
    \multirow{2}{*}{$8$}   & \multirow{2}{*}{$1.9536$}          & $0.9223$     & $0.94139$  & $0.064404$   & $20$\\ \cline{3-6}
                           &                                   & $1.9418$       & $9.5627$   & $0.0000037429$      & $1000$\\ \hline
    \multirow{2}{*}{$9$}   & \multirow{2}{*}{$1.6639$}         & $1.4446$     & $1.9372$  & $0.048915$   & $20$ \\ \cline{3-6}
                           &                                    & $1.7266$       & $16.738$    & $0.0019792$    & $1000$\\ \hline
    \multirow{2}{*}{$10$}  & \multirow{2}{*}{$1.4293$}          & $2.0005$     & $2.0226$  & $0.016453$   & $20$   \\ \cline{3-6}
                           &                                   & $1.4917$       & $16.035$    & $0.00015252$     & $1000$\\ \hline
 \multicolumn{2}{|c|}{\multirow{2}{*}{Uniform Sample} }         & $26.428$     & $641.59$ & $0.085716$     & $20$\\ \cline{3-6}
 \multicolumn{1}{|c}{}     &                                    & $11.905$     & $256.0$ & $0.010946$  & $1000$\\ \hline
  \end{tabular}
  \end{table}

\begin{table}
\caption{Sampling results for the Three-Hump Camel Function (Modified-S) over the simplex}
\label{table:samplerd}
  \begin{tabular}{| c | c | c | c | c |c|}
\hline
     r                     & $\underline{f}_{\mathbf{K}}^{(r)}$ & Mean          & Variance     & Minimum    & Sample Size\\ \hline
    \multirow{2}{*}{$1$}   & \multirow{2}{*}{$84.354$}        & $104.93$     & $122488.0$  & $0.33441$ & $20$ \\ \cline{3-6}
                           &                                  & $89.732$       & $48238.0$  & $0.0011036$  & $1000$\\ \hline
    \multirow{2}{*}{$2$}   & \multirow{2}{*}{$22.398$}         & $37.036$     & $9864.0$  & $0.57012$    & $20$\\ \cline{3-6}
                           &                                    & $22.292$       & $10102.0$  & $0.0022204$    & $1000$\\ \hline
    \multirow{2}{*}{$3$}   & \multirow{2}{*}{$12.353$}       & $3.4161$     & $49.898$  & $0.28108$  & $20$ \\ \cline{3-6}
                           &                                    & $11.707$       & $1515.9$  & $0.00065454$     & $1000$\\ \hline
    \multirow{2}{*}{$4$}   & \multirow{2}{*}{$3.9153$}        & $2.4193$     & $9.0182$  & $0.16865$  & $20$  \\ \cline{3-6}
                           &                                    & $3.6768$       & $592.96$  & $0.0016775$   & $1000$\\ \hline
    \multirow{2}{*}{$5$}   & \multirow{2}{*}{$2.9782$}        & $1.8336$     & $6.3414$   & $0.11311$   & $20$ \\ \cline{3-6}
                           &                                    & $2.5237$       & $47.619$   & $0.00097905$    & $1000$\\ \hline
    \multirow{2}{*}{$6$}   & \multirow{2}{*}{$1.3303$}       & $2.355$     & $26.176$  & $0.0092016$    & $20$\\ \cline{3-6}
                           &                                    & $1.2134$       & $8.7253$   & $0.00040725$    & $1000$\\ \hline
    \multirow{2}{*}{$7$}   & \multirow{2}{*}{$1.1773$}        & $1.0385$     & $1.0569$   & $0.053695$    & $20$ \\ \cline{3-6}
                           &                                    & $1.092$       & $6.718$   & $0.00050329$    & $1000$\\ \hline
    \multirow{2}{*}{$8$}   & \multirow{2}{*}{$0.77992$}        & $0.9737$     & $0.73522$  & $0.10604$   & $20$\\ \cline{3-6}
                           &                                    & $0.72927$       & $0.73641$   & $0.00048517$      & $1000$\\ \hline
    \multirow{2}{*}{$9$}   & \multirow{2}{*}{$0.73202$}      & $0.69755$     & $0.19107$  & $0.051634$   & $20$ \\ \cline{3-6}
                           &                                    & $0.65302$       & $0.28537$    & $0.00024601$    & $1000$\\ \hline
    \multirow{2}{*}{$10$}  & \multirow{2}{*}{$0.60846$}        & $0.67575$     & $0.17453$  & $0.010351$   & $20$   \\ \cline{3-6}
                           &                                    & $0.5616$       & $0.17821$    & $0.00044175$     & $1000$\\ \hline

 \multicolumn{2}{|c|}{\multirow{2}{*}{Uniform Sample} }         & $518.48$     & $354855.0$ & $0.9165$     & $20$\\ \cline{3-6}
 \multicolumn{1}{|c}{}     &                                    & $485.77$     & $391577.0$ & $0.32713$  & $1000$\\ \hline
  \end{tabular}
  \end{table}

\noindent Note that the average of the sample function values approximate $\underline{f}_{\mathbf{K}}^{(r)}$ reasonably well for sample size $1000$, but poorly for sample
size $20$. Moreover, the average sample function value for uniform sampling from $\bK$ is much higher than $\underline{f}_{\mathbf{K}}^{(r)}$.
Also, the minimum function value for sampling from $\mathcal{T}_{\mathbf{K}}$ is significantly lower than the minimum function value
obtained by uniform sampling for most values of $r$.
In terms of generating ``good" feasible solutions, sampling from  $\mathcal{T}_{\mathbf{K}}$ therefore outperforms uniform sampling from $\bK$ for these examples, as
one would expect.

\section{Concluding remarks}
{We conclude with some additional remarks on Assumption \ref{propeta}, and some
discussion on perspectives for future work.
}

\label{sec:conclusion}
\subsection{Revisiting Assumption \ref{propeta}}\label{secass}

\noindent In this section we consider in more detail Assumption \ref{propeta}, the geometric assumption which we made about the set $\mathbf K$.
 First we recall another condition, known as the {\em interior cone condition},  which is classically used in approximation theory (see, e.g., Wendland \cite{HW01}).

\begin{definition}\label{deficc}
\cite[Definition 3.1]{HW01} A set $\mathbf{K}\subseteq \oR^n$ is said to satisfy an interior cone condition if there exist an angle $\theta\in (0,{\pi/ 2})$ and a radius $\rho>0$ such that, for every $x\in\mathbf{K}$, a unit vector $\xi(x)$ exists such that the set
\begin{equation}\label{eqcxp}
C(x,\xi(x),\theta,\rho):=\{x+\lambda y: y\in\oR^n,\|y\|=1,y^T\xi(x)\ge\cos{\theta},\lambda\in[0,\rho] \}
\end{equation}
is contained in $\mathbf{K}$.
\end{definition}

\smallskip\noindent
For instance, as  we now recall, Euclidean balls and star-shaped sets satisfy the interior cone condition.

\begin{lemma}\cite[Lemma 3.10]{HW01}\label{lemball}
Every Euclidean ball with radius $r>0$ satisfies an interior cone condition with radius $\rho=r$ and angle $\theta=\pi/3$.
\end{lemma}

\begin{definition}\label{defstarshaped}
\cite[Definition 11.25]{HW01}
A set $\bK$ is said to be {\em star-shaped} with respect to a ball $B_r(x_c)$ if, for every $x\in\bK$, the closed convex hull of $\{x\}\cup B_r(x_c)$ is contained in $\bK$.
\end{definition}

\begin{proposition}\cite[Proposition 11.26]{HW01}\label{propstar}
If $\bK$ is bounded, star-shaped with respect to a ball $B_r(x_c)$, then $\bK$ satisfies an interior cone condition with radius $\rho=r$ and angle $\theta=2\arcsin\left[{r\over 2\sqrt{D(\bK)}}\right]$.
\end{proposition}

\medskip
\noindent
{In fact, any set satisfying the interior cone condition also satisfies the following stronger version of Assumption \ref{propeta}.}

\begin{assumption}\label{ass2}
{There exist constants $\eta_{\mathbf{K}}>0$ and ${\epsilon}_{\mathbf{K}}>0$ such that, for all points $a\in \mathbf K$,}
\begin{equation}\label{volballn2}
\vol (B_\epsilon(a)\cap \mathbf{K}) \ge {\eta_{\mathbf{K}}}\vol B_\epsilon(a)={\eta_{\mathbf{K}}}\epsilon^n\gamma_n \ \ \text{for all $0<\epsilon\le{\epsilon}_{\mathbf{K}}$}.
\end{equation}
\end{assumption}

\smallskip\noindent
{Hence the only difference with Assumption \ref{propeta} is that  the constants $\eta_{\mathbf K}$ and $\epsilon_{\mathbf K}$ now depend only on the set $\mathbf K$ and not on the choice of $a\in \mathbf K$. Clearly, Assumption \ref{ass2} implies Assumption \ref{propeta}. Moreover, any set satisfying the interior cone condition satisfies Assumption \ref{ass2}.}

\begin{lemma}\label{claimiccass}
If a  set $\mathbf{K}\subseteq \oR^n$ satisfies the  interior cone condition (\ref{eqcxp}) then $\bK$ also satisfies Assumption \ref{ass2} (and thus Assumption \ref{propeta}), where we set
\begin{equation*}\label{etakepsilonk}
\eta_{\mathbf{K}}=\left[ \sin{\theta}\over 1+\sin{\theta} \right]^n\ \ \text{and} \ \ \epsilon_{\mathbf{K}}=\rho.
\end{equation*}
\end{lemma}


\begin{proof} 
Assume  that $\mathbf{K}$ satisfies the interior cone condition (\ref{eqcxp}).
Then, using \cite[Lemma 3.7]{HW01}, we know that,  for every $x\in \bK$ and $h\le \rho/(1+\sin{\theta})$, the closed ball $B_{h\sin{\theta}}(x+h\xi(x))$ is contained in $C(x,\xi(x),\theta,\rho)$ and thus in $ \bK$.
Then,
for all $x_0\in\mathbf{K}$ and $\epsilon\in(0,\rho]$, after setting $h=\epsilon/(1+\sin{\theta})$, one can obtain
\begin{equation*}
{\vol (B_{\epsilon}(x_0)\cap \mathbf{K}) \over \vol B_{\epsilon}(x_0)}\ge {\vol C(x_0,\xi(x_0),\theta,\epsilon) \over \vol B_{\epsilon}(x_0)}\ge {\vol B_{h\sin{\theta}}(x_0+h\xi(x_0)) \over \vol B_{\epsilon}(x_0)}=\left[ \sin{\theta}\over 1+\sin{\theta} \right]^n.
\end{equation*}
Thus, Assumption \ref{ass2} holds after setting  $\eta_{\mathbf{K}}=\left[ \sin{\theta}\over 1+\sin{\theta} \right]^n$ and $\epsilon_\bK=\rho$.
\qed
\end{proof}

\smallskip
\noindent
As any convex body (i.e., full-dimensional convex and compact)  is star-shaped with respect to any ball it contains, the next result follows as a direct application of Proposition \ref{propstar} and Lemma \ref{claimiccass}.

\begin{corollary}\label{corptass}
{Any convex body satisfies the interior cone condition and thus Assumptions \ref{propeta} and \ref{ass2}.}
\end{corollary}

\smallskip\noindent
{As an illustration we now consider the parameters $\eta_{\mathbf K}$, $\epsilon_{\mathbf K}$, and $r_{\mathbf K}$ (from relation (\ref{addefrK})) when $\mathbf K$ is the hypercube, the simplex and the Euclidean ball.}

\begin{remark}\label{exrK}
{Consider first the case when $\mathbf K$ is the  hypercube\index{Hypercube} $\oQ_n=[0,1]^n$.
By Proposition \ref{propstar}, it satisfies the interior cone condition with radius $\rho=1/2$ and angle $\theta=2\arcsin\left[{1\over 4\sqrt{n}}\right]$.
Hence, Assumption \ref{ass2} holds with $\epsilon_{\mathbf K}=1/2$ and
$\eta_{\mathbf K}= \left({\sqrt{16n-1} \over 8n + \sqrt{16n-1}}\right)^n$ (which is $\sim \left({1\over 2\sqrt n}\right)^n$ for $n$ large).
Moreover, as $D(\mathbf K)=n$, it follows that
$r_{\mathbf K}= 4ne$.}

\smallskip\noindent
{Consider now the case when $\mathbf K$ is the full-dimensional simplex ${\Delta_n}$. By Proposition \ref{propstar},  it  satisfies the interior cone condition with radius
$\rho={1\over n+\sqrt{n}}$ and angle $\theta=2 \arcsin\left[{1\over 2\sqrt{2}(n+\sqrt{n})}\right]$
(since the ball with center $\rho(1,\ldots,1)^T$ and radius $\rho$ is contained in $\widehat \Delta_n$).
Hence Assumption \ref{ass2} holds with $\epsilon_{\mathbf K}={1\over n+\sqrt{n}}$
and $\eta_{\mathbf K}= \left({ \sqrt{8(n+\sqrt n)^2-1} \over 4(n+\sqrt n)^2 +\sqrt{ 8(n+\sqrt n)^2 -1}} \right)^n$ (which is $ \sim \left({1\over \sqrt 2 n}\right)^n$ for $n$ large). As $D(\mathbf K)= 2$, it follows that $r_{\mathbf K}= e(n+\sqrt n)^3$.}

{Finally, for the Euclidean ball $\mathbf K=B_1(0)$, we have $\epsilon_{\mathbf K}=1$,
$\eta_{\mathbf K}=\left({\sqrt 3\over 2+\sqrt 3}\right)^n$ and
$r_{\mathbf K}=\max\{2e,n\}.$
}
\end{remark}

\subsection{Perspectives}\label{secrandom}

The sampling approach of Section \ref{secgenerate} often provides good feasible solutions for the examples in Section \ref{sec:numerial examples},
even for small values of $r$. One may therefore explore using the sampling technique (for small $r$) as a way of generating
starting points for multi-start global optimization algorithms.

Another possibility to enhance computation would be to investigate other sufficient conditions for nonnegativity of $h$ on $\bK$, {more general}  than the sum-of-squares condition
studied here. This may result in a faster rate of convergence than for $\underline{f}^{(r)}_{\mathbf{K}}$.

Finally, understanding the exact rate of convergence of the upper bounds $\underline f^{(r)}_{\mathbf K}$ remains an open problem.
In particular we do not know whether $1/\sqrt r$ is the right rate of convergence.

\section*{Acknowledgements}
We thank Jean Bernard Lasserre for bringing our attention to his work \cite{Las11} and for several valuable suggestions, and Dorota Kurowicka for valuable discussions on multivariate sampling techniques.

\end{document}